\documentclass[EJP]{ejpecp} % replace ECP by EJP if needed.  

\usepackage{graphics} 
\usepackage{url}

\SHORTTITLE{} 

\TITLE{Viscosity methods giving uniqueness for  martingale problems\thanks{Research supported in part by NSF grant DMS 11-06424}} % \thanks is optional. Insert line breaks with \\

\AUTHORS{%
  Cristina~Costantini\footnote{Dipartimento di Economia, Universit\`a di Chieti-Pescara, Italy. 
    \EMAIL{c.costantini@unich.it}}
  \and %% remove this line and below if single author
  Thomas~G.~Kurtz\footnote{ Departments of Mathematics and Statistics, University of                     Wisconsin - Madison, USA. 
    \EMAIL{kurtz@math.wisc.edu,} \url{http://www.math.wisc.edu/\~kurtz/}}}%AUTHORS
 % Separate items with ;

\AMSSUBJ{60J25; 60J35} % Edit. Separate items with ;
\AMSSUBJSECONDARY{60G46; 47D07} % Optional, separate items with ;

\SUBMITTED{June 25, 2014} % Edit.
\ACCEPTED{May 31, 2015} % Edit.

\ARXIVID{1406.6650v1} % Edit.
\VOLUME{20}
\YEAR{2015}
\PAPERNUM{67}
\DOI{v20-3624}

\ABSTRACT{Let $E$ be a complete, separable metric space and 
$A$ be an operator on $C_b(E)$.  
We give an abstract definition of viscosity 
sub/supersolution of the resolvent equation $\lambda u-Au=h$ 
and show that, if 
the comparison principle holds, then the martingale 
problem for $A$ has a unique solution. Our proofs work 
also under two alternative definitions 
of viscosity sub/supersolution which might be useful, in 
particular, in infinite dimensional spaces, for instance to 
study measure-valued processes. 

We prove the analogous result for stochastic processes 
that must satisfy boundary conditions, modeled as solutions of 
constrained martingale problems. In the case of 
reflecting diffusions in $D\subset {\bf R}^d$, our assumptions allow $
D$ 
to be nonsmooth and the direction of reflection to be 
degenerate. 

Two examples are presented: A diffusion with degenerate 
oblique direction 
of reflection and a class of jump diffusion processes with infinite 
variation jump component and possibly degenerate 
diffusion matrix. }

\def \hat{\widehat}
\def \tilde{\widetilde}
\def \bar{\overline}

\begin{document}

%%%%%%%%%%%%%%%%%%%%%%%%%%%%%%%%%%%%%%%%%%%%%%%%%%%%%%%%%%%%%%%%%%%
%%                                                               %%
%% No need for \maketitle.                                       %%
%%                                                               %%
%%%%%%%%%%%%%%%%%%%%%%%%%%%%%%%%%%%%%%%%%%%%%%%%%%%%%%%%%%%%%%%%%%%

%%%%%%%%%%%%%%%%%%%%%%%%%%%%%%%%%%%%%%%%%%%%%%%%%%%%%%%%%%%%%%%%%%%
%%                                                               %%
%% Please replace what follows by the body of your article       %%
%% (up to the bibliography):                                     %%
%%                                                               %%
%%%%%%%%%%%%%%%%%%%%%%%%%%%%%%%%%%%%%%%%%%%%%%%%%%%%%%%%%%%%%%%%%%%

\section{Introduction}

There are many ways of specifying Markov processes, 
the most popular being as solutions of stochastic 
equations, as solutions of martingale problems, or in 
terms of solutions of the Kolmogorov forward equation 
(the Fokker-Planck equation or the master equation 
depending on context).  The solution of a stochastic 
equation explicitly gives a process while a solution 
of a martingale problem gives the distribution of a
process and a solution of a forward equation gives the 
one dimensional distributions of a process.
Typically, these approaches are 
equivalent (assuming that there is a stochastic equation 
formulation) in the sense that existence of a solution 
specified by one method implies existence of 
corresponding solutions to 
the other two (weak existence for the stochastic 
equation) and hence uniqueness for one method implies 
uniqueness for the other two (distributional uniqueness 
for the stochastic equation).

One approach to proving uniqueness for a forward 
equation and hence for the corresponding martingale 
problem is to verify a condition on the generator similar 
to the range condition of the Hille-Yosida theorem. (See 
Corollary \ref{rngunq}.)  
We show that the original generator $A$ of our martingale 
problem (or a restriction of the original generator $A$, 
in the case of martingale 
problems with boundary conditions) can be extended 
to a generator $\hat {A}$ such that every solution of 
the martingale problem for $A$ is a solution for $\hat {A}$ and $
\hat {A}$ 
satisfies the range condition of Corollary \ref{rngunq}.  
Our extension is constructed by including viscosity solutions of 
the resolvent equation 
\begin{equation}\lambda u-Au=h.\label{ireso}\end{equation}

Viscosity solutions have been used to study value 
functions in stochastic optimal 
control and optimal stopping theory since the very 
beginning (see the classical references \cite{CIL92}, \cite{Ph98},   
as well as \cite{FS06}). It may be interesting to note 
that, in the context of 
Hamilton-Jacobi equations, the idea of studying a parabolic 
equation by solving a resolvent 
equation in the viscosity sense appears already in \cite{CL83}, Section VI.3, 
where it is applied to a model problem. 
The methodology is also important for related 
problems in finance (for example \cite{ST00},  
\cite{BKR01}, \cite{KK04}, \cite{BCM10}, \cite{CPD12} and 
many others).

Viscosity solutions have also been used to study the 
partial differential equations associated with
forward-backward stochastic differential equations 
(\cite{MZ11},\cite{EKTZ14}) and in the theory of large deviations (\cite{FK06}). 

The basic data for our work is an operator 
$A\subset C_b(E)\times C_b(E)$ on a complete, separable metric space 
$E$.  We offer an abstract definition of viscosity 
sub/supersolution for (\ref{ireso}) (which for 
integro-differential operators in ${\bf R}^d$ is equivalent to the 
usual one) and prove, under very general conditions, that 
the martingale problem for $A$ has a unique solution if 
the comparison principle for $(\ref{ireso})$ holds.  

We believe the interest of this result is twofold: on one 
hand it clarifies the general connection between 
viscosity solutions and martingale problems; on the 
other, there are still many martingale problems, for 
instance in infinite dimension, for which 
uniqueness is an open question. 

We also discuss two alternative abstract definitions 
of viscosity sub/supersolution that might be especially 
useful in infinite dimensional spaces. 
All our proofs work under these alternative definitions 
as well.

The first alternative definition is a modification of 
a definition suggested to us by Nizar Touzi and used in 
\cite{EKTZ14}. Being a stronger definition (it allows 
for more test functions), it should be easier to prove 
comparison results under this definition. 

The second alternative definition appears in \cite{FK06} 
and is a stronger definition too.   
Under this definition, a sort of converse of our main 
result holds, namely if $h$ belongs to $\overline {{\cal R}(\lambda 
-A)}$ (under 
uniform convergence), then the comparison principle for 
semisolutions of (\ref{ireso}) holds (Theorem 
\ref{converse}).  When $E$ is compact, 
this definition is equivalent to our main definition, 
hence the comparison principle holds for semisolutions 
in that sense as well.  

Next we consider stochastic processes that must satisfy 
some boundary conditions, for example, reflecting 
diffusions.  Boundary conditions are expressed in terms 
of an operator $B$ which enters into the formulation of a 
{\em constrained martingale problem\/} (see \cite{Kur90}).  We 
restrict our attention to models in which the boundary 
term in the constrained martingale problem is expressed 
as an integral against a local time.  Then it still holds 
that uniqueness of the solution of the constrained 
martingale problem follows from the comparison 
principle between viscosity sub and supersolutions of 
(\ref{ireso}) with the appropriate boundary conditions.  
Notice that, as for the standard martingale problem, 
uniqueness for the constrained martingale problem 
implies that the solution is Markovian (see \cite{Kur90}, 
Proposition 2.6).  
  
In the presence of boundary conditions, even for 
${\bf R}^d$-valued diffusions, there are examples for which 
uniqueness of the martingale problem is not known.  
Processes in domains with boundaries that are only 
piecewise smooth or with boundary operators that are 
second order or with directions of reflection that are 
tangential on some part of the boundary continue to be a 
challenge.  In this last case, as an example of an application 
of our results, we use the comparison principle proved 
in \cite{PK05} to obtain uniqueness.  

The strategy of our proofs has been initially inspired by the proof 
of Krylov's selection theorem for martingale problems 
that appears in \cite{EK86} 
and originally appeared in unpublished work of 
\cite{GG77}. In that proof the generator is 
recursively extended in such a way that 
there are always solutions of the martingale problem
for the extended generator, but eventually only one. 
If uniqueness fails for the original martingale problem, 
there is more than one way 
to do the extension. Conversely if, at each stage of the 
recursion, there is only one 
way to do the extension and all solutions of the 
martingale problem for the original generator remain 
solutions for the extended generator, then 
uniqueness must hold for the original generator.

Analogously, assuming the comparison principle 
for (\ref{ireso}) (or (\ref{ireso}) with the appropriate 
boundary conditions) holds for a large enough class of 
functions $h$, we construct an 
extension $\hat {A}$ of the original operator $A$ 
(of a restriction of the original operator $A$, 
in the case of constrained martingale problems) 
such that all solutions of the martingale problem (the  
constrained martingale problem) for $A$ 
are solutions of the martingale problem for $\hat {A}$, and such 
that uniqueness holds for $\hat {A}$. Actually, in the case of 
ordinary martingale problems the extension, although 
possible, is not needed, because 
the comparison principle for $\mbox{\rm (\ref{ireso})}$ directly yields a condition 
($(\ref{lapid})$) that, if valid for a large enough class of 
functions $h$, ensures uniqueness of the one-dimensional distributions 
of solutions to the martingale problem, and hence 
uniqueness of the solution. The extension is needed, 
instead, for constrained martingale problems. 

A few works on viscosity solutions of partial 
differential equations 
and weak solutions of stochastic differential equations 
have appeared in recent years. 
For diffusions in ${\bf R}^d$, \cite{BS12}, assuming a comparison 
principle exists, 
show that the backward equation has a
unique viscosity solution, and it follows that the 
corresponding stochastic differential equation has a 
unique weak solution. For Markovian forward-backward 
stochastic differential equations, \cite{MZ11} also derive 
uniqueness of the weak solution from existence of a 
comparison principle for the corresponding partial 
differential equation. In the non-Markovian case, the 
associated partial differential equation becomes path 
dependent. \cite{EKTZ14} propose the notion of viscosity 
solution of (semilinear) path dependent partial 
differential equations on the space of 
continuous paths already mentioned above and prove a comparison principle.  

The rest of this paper is organized as follows: Section 
\ref{backgrnd} contains some background material on 
martingale problems and on viscosity solutions; Section 
\ref{mgpunq} deals with martingale problems; 
the alternative definitions of viscosity solution are 
discussed in Section \ref{altdef}; 
Section \ref{cmgpunq} deals with martingale 
problems with boundary conditions; finally in Section 
\ref{examples}, we present two examples, including the 
application to diffusions with degenerate direction of reflection.

\section{Background}\label{backgrnd}

\subsection{Martingale problems}

Throughout this paper we will assume that $(E,r)$ is a 
complete separable metric space, $D_E[0,\infty )$ is the space of 
cadlag, $E$-valued functions endowed with the Skorohod topology, 
${\cal B}(D_E[0,\infty )))$ is the $\sigma$-algebra of Borel sets of $
D_E[0,\infty )$, 
$C_b(E)$ denotes the space of bounded, continuous 
functions on $(E,r)$,  $B(E)$ denotes the space of 
bounded, measurable functions on $(E,r)$,
 and ${\cal P}(E)$ denotes the space of 
probability measures on $(E,r)$. $||\cdot ||$ will denote the 
supremum norm on $C_b(E)$ or $B(E)$.

\begin{definition}\label{mgp}
A measurable 
stochastic process $X$ defined on a probability 
space $(\Omega ,{\cal F},{\mathbb P})$ is a solution of 
the {\em martingale problem\/} for
\[A\subset B(E)\times B(E),\]
provided there exists a filtration $\{{\cal F}_t\}$ such that $X$ and 
$\int_0^{\dot{}}g(X(s))ds$ are $\{{\cal F}_t\}$-adapted, for every $
g\in B(E)$, and 
\begin{equation}M_f(t)=f(X(t))-f(X(0))-\int_0^tg(X(s))ds\label{mg}\end{equation}
is a $\{{\cal F}_t\}$-martingale for each $(f,g)\in A$.  If $X(0)$ has distribution $
\mu$, 
we say $X$ is a solution of the martingale problem for $(A,\mu )$.
\end{definition}\ 

\begin{remark}
Because linear combinations of martingales are martingales, 
without loss of generality, we can, but need not, 
assume that $A$ is linear and that $(1,0)\in A$.

We do not, and cannot for our purposes, require $A$ to be 
single valued.  In particular, the operator $\hat {A}$ defined in 
the proof of Theorem \ref{cmgpb} will typically not be 
single valued.
\end{remark}\medskip

In the next sections 
we will restrict our attention to 
processes $X$ with sample paths in $D_E[0,\infty )$. 

\begin{definition}\label{Dmgp}
A stochastic process $X$, defined on a probability 
space $(\Omega ,{\cal F},{\mathbb P})$, with sample paths in $D_E[0,
\infty )$ is a solution of 
the {\em martingale problem\/} for
\[A\subset B(E)\times B(E),\]
{\em in} $D_E[0,\infty )$ provided there exists a filtration $\{{\cal F}_
t\}$ such 
that $X$ is $\{{\cal F}_t\}$-adapted and 
\begin{equation}M_f(t)=f(X(t))-f(X(0))-\int_0^tg(X(s))ds\label{Dmg}\end{equation}
is a $\{{\cal F}_t\}$-martingale for each $(f,g)\in A$.  If $X(0)$ has distribution $
\mu$, 
we say $X$ is a solution of the martingale problem for 
$(A,\mu )$ in $D_E[0,\infty )$.
\end{definition}\ 

\begin{remark}
Since $X$ has sample paths in $D_E[0,\infty )$, the fact that 
it is $\{{\cal F}_t\}$-adapted implies that 
$\int_0^{\dot{}}g(X(s))ds$ is $\{{\cal F}_t\}$-adapted, for every $
g\in B(E)$. 
\end{remark}\medskip

\begin{remark}\label{cadlag}
The requirement that $X$ have sample paths in $D_E[0,\infty )$ is 
usually fulfilled provided the state space $E$ is selected 
appropriately.
Moreover, if $A\subset C_b(E)\times C_b(E)$, ${\cal D}(A)$ is dense in $
C_b(E)$ in the topology of 
uniform convergence on compact sets and 
for each compact $K\subset E$, $\varepsilon >0$, and $T>0$, there 
exists a compact $K'\subset E$ such that 
\[P\{X(t)\in K',t\leq T,X(0)\in K\}\geq (1-\varepsilon )P\{X(0)\in 
K\},\]
for every progressive process such that (\ref{mg}) is a martingale, 
then every such process  
has a modification with sample paths in $D_E[0,\infty )$ (See 
Theorem 4.3.6 of \cite{EK86}.) 
\end{remark}\medskip

\begin{remark}\label{fdd}
The martingale property is equivalent to the 
requirement that
\[E[(f(X(t+r))-f(X(t))-\int_t^{t+r}g(X(s))ds)\prod_ih_i(X(t_i))]=
0\]
for all choices of $(f,g)\in A$, $h_i\in B(E)$, $t,r\geq 0$, and 
$0\leq t_i\leq t$.  Consequently, the property of being a solution 
of a martingale problem is a property of the finite-
dimensional distributions of $X$. 

In particular, for the martingale problem in $D_E[0,\infty )$, 
the property of being a solution is a property of the 
distribution of $X$ on $D_E[0,\infty )$. 
Much of what follows in the next sections will be 
formulated in terms of the collection $\Pi\subset {\cal P}(D_E[0,
\infty ))$ of 
distributions of solutions of the martingale problem. For 
some purposes, it will be convenient to assume that $X$ is 
the canonical process defined on 
$(\Omega ,{\cal F},\mathbb{P})=(D_E[0,\infty ),{\cal B}(D_E[0,\infty 
))),P)$, for some 
$P\in\Pi$. 
\end{remark}\medskip

In view of Remark \ref{fdd} it is clear that 
uniqueness of the solution to the martingale problem for 
an operator $A$ is to be meant as uniqueness of the 
finite-dimensional distributions. 

\begin{definition}\label{uniq}
We say that {\em uniqueness holds for the martingale }
{\em problem for} $A$  if, for every $
\mu$, any two solutions of the martingale 
problem for $(A,\mu )$  have the same finite-dimensional 
distributions. If we restrict our attention to solutions in $D_E[0,\infty)$, then uniqueness holds if any two solutions for $(A,\mu)$ have the same distribution on $D_E[0,\infty)$
\end{definition}

One of the most important consequences of the 
martingale approach to Markov processes is that 
uniqueness of one-dimensional distributions implies 
uniqueness of finite-dimensional distributions.

\begin{theorem}\label{1dfid}
Suppose that for each $\mu\in {\cal P}(E)$, 
any two solutions of the martingale 
problem for $(A,\mu )$  have the same one-dimensional 
distributions.  Then any two solutions have the same 
finite-dimensional distributions. If any two solutions of the martingale 
problem for $(A,\mu )$  in $D_E[0,\infty)$ have the same one-dimensional 
distributions, then they have the same 
distribution on $D_E[0,\infty)$.
\end{theorem}

\begin{proof}
This is a classical result. See for instance Theorem 4.4.2 
and Corollary 4.4.3 of \cite{EK86}.
\end{proof}

For $\mu\in {\cal P}(E)$ and $f\in B(E)$ we will use the notation 
\begin{equation}\mu f=\int_Ef(x)\mu (dx).\label{notat-i}\end{equation}

\begin{lemma}\label{resid}Let $X$ be a $\{{\cal F}_t\}$-adapted stochastic process 
with sample paths in $D_E[0,\infty )$, with initial distribution 
$\mu$, $f,g\in B(E)$ and $\lambda >0$. Then $(\mbox{\rm \ref{Dmg}}
)$ is 
a $\{{\cal F}_t\}$-martingale if and only if 
\begin{equation}M_f^{\lambda}(t)=e^{-\lambda t}f(X(t))-f(X(0))+\int_0^te^{-\lambda 
s}(\lambda f(X(s))-g(X(s))ds\label{resmg}\end{equation}
is a $\{{\cal F}_t\}$-martingale. In particular, if $(\mbox{\rm \ref{Dmg}}
)$ is a $\{{\cal F}_t\}$-martingale 
\begin{equation}\mu f=E[\int_0^{\infty}e^{-\lambda s}(\lambda f(X
(s))-g(X(s)))ds].\label{inrg}\end{equation}
\end{lemma}

\begin{proof}\ The general statement is a special case of Lemma 
4.3.2 of \cite{EK86}. If $f$ is  
continuous, as will typically be the case in the next sections, $M_f$ will be cadlag and
we can apply It\^o's formula to obtain 
\[e^{-\lambda t}f(X(t))-f(X(0))=\int_0^t(-f(X(s))\lambda e^{-\lambda 
s}+e^{-\lambda s}g(X(s)))ds+\int_0^te^{-\lambda s}dM_f(s),\]
where the last term on right is a 
$\{{\cal F}_t\}$-martingale. (Note that, since all the processes involved are cadlag, 
we do not need to require the filtration $\{{\cal F}_
t\}$ 
to satisfy the `usual conditions'.) 
Conversely, if $(\mbox{\rm \ref{resmg}})$ is a 
$\{{\cal F}_t\}$-martingale, the assertion follows by applying It\^o's 
formula to $f(X(t))=e^{\lambda t}\Big(e^{-\lambda t}f(X(t))\Big)$. 

In particular, if $(\mbox{\rm \ref{Dmg}})$ is a $\{{\cal F}_t\}$-martingale 
\[E[f(X(0)]-E[e^{-\lambda t}f(X(t))]=E[\int_0^te^{-\lambda s}(\lambda 
f(X(s))-g(X(s))ds],\]
and the second statement follows by taking $t\rightarrow\infty$. 
\end{proof}

\begin{lemma}\label{rnlem}
Let $X$ be a solution of the martingale problem for $A\subset C_b(E)\times C_b(E)$ in 
$D_E[0,\infty )$ with respect to a filtration $\{{\cal F}_t\}$.
Let $\tau\geq 0$ be a finite $\{{\cal F}_t\}$-stopping time and 
$H\geq 0$ be a ${\cal F}_{\tau}$-measurable random variable such that 
$0<E[H]<\infty$. Define $P^{\tau ,H}$ by 
\begin{equation}P^{\tau ,H}(C)=\frac {E[H{\bf 1}_C(X(\tau +\cdot 
))]}{E[H]},\quad C\in {\cal B}(D_E[0,\infty )).\label{phdist}\end{equation}
Then $P^{\tau ,H}$ is the distribution of a solution of the martingale 
problem for $A$ in $D_E[0,\infty )$.
\end{lemma}

\begin{proof}
Let $(\Omega ,{\cal F},\mathbb{P})$ be the probability space 
on which $X$ is defined, and define $\mathbb{P}^H$ on $(\Omega ,{\cal F}
)$ by 
\[\mathbb{P}^H(C)=\frac {E^{\mathbb{P}}[H{\bf 1}_C]}{E^{\mathbb{P}}
[H]},\quad C\in {\cal F}.\]
Define  $X^{\tau}$ by $X^{\tau}(t)=X(\tau +t)$. 
$X^{\tau}$ is adapted to the filtration $\{{\cal F}_{\tau +t}\}$ and 
for  $0\leq t_1<\cdots <t_n<t_{n+1}$ and  $f_1,\cdots ,f_n\in B(E
)$,
\begin{eqnarray*}
&&E^{\mathbb{P}^H}\Big[\Big\{f(X^{\tau}(t_{n+1}))-f(X^{\tau}(t_n)
)-\int_{t_n}^{t_{n+1}}Af(X^{\tau}(s))ds\Big\}\,\Pi_{i=1}^nf_i(X^{
\tau}(t_i))\Big]\qquad\qquad\qquad\qquad\qquad\qquad\qquad\qquad\\
&&\quad =\frac 1{E^{\mathbb{P}}[H]}E^{\mathbb{P}}\Big[H\Big\{f(X(
\tau +t_{n+1}))-f(X(\tau +t_n))\\
&&\qquad\qquad\qquad\qquad -\int_{\tau +t_n}^{\tau +t_{n+1}}Af(X(
s))ds\Big\}\,\Pi_{i=1}^nf_i(X(\tau +t_i))\Big]\qquad\qquad\qquad\qquad
\qquad\qquad\\
&&\quad =0\end{eqnarray*}
by the optional sampling theorem.
Therefore, under $\mathbb{P}^H$, $X^{\tau}$ is a solution of the 
martingale problem. $P^{\tau ,H}$, given by (\ref{phdist}),
is the distribution of $X^{\tau}$ on $D_E[0,\infty )$.
\end{proof}

\begin{lemma}\label{ext} Let $\lambda >0$. Suppose $u,h\in B(E)$ satisfy 
\begin{equation}\mu u=E[\int_0^{\infty}e^{-\lambda t}h(X(t))dt],\label{mgcond}\end{equation}
for every solution of the martingale problem for $A$  in $D_E[0,\infty )$ with 
initial distribution $\mu$ and for every $\mu\in {\cal P}(E)$.
Then 
\begin{equation}u(X(t))-\int_0^t(\lambda u(X(s))-h(X(s)))ds\label{mgprop}\end{equation}
is a $\left\{{\cal F}_t^X\right\}$-martingale for every solution $
X$ of the martingale 
problem for $A$ in $D_E[0,\infty )$.
\end{lemma}

\begin{proof}\ The lemma is Lemma 4.5.18 of \cite{EK86}, 
but we repeat the proof here for the convenience of the 
reader. 

Let $X$ be a solution of the martingale problem for $A$ on a 
probability space $(\Omega ,{\cal F},\mathbb{P})$. For $t\geq 0$ and 
$B\in {\cal F}_t$ with $\mathbb{P}(B)>0$, define $\mathbb{Q}$ on $
(\Omega ,{\cal F})$ by 
\[\mathbb{Q}(C)=\frac {E^{\mathbb{P}}[{\bf 1}_B{\bf 1}_C]}{\mathbb{
P}(B)},\quad C\in {\cal F}.\]
Then 
\begin{eqnarray*}
&&E^{\mathbb{P}}[{\bf 1}_Be^{\lambda t}\int_t^{\infty}e^{-\lambda 
s}h(X(s))ds]=E^{\mathbb{P}}[{\bf 1}_B\int_0^{\infty}e^{-\lambda s}
h(X(t+s))ds]\\
&&=\mathbb{P}(B)E^{\mathbb{Q}}[\int_0^{\infty}e^{-\lambda s}h(X(t
+s))ds].\end{eqnarray*}
By the same arguments as in the proof of Lemma 
\ref{rnlem}, $X(t+\cdot )$ is a solution of the martingale 
problem for $A$ on the probability space $(\Omega ,{\cal F},\mathbb{
Q})$ 
with  respect to the filtration $\{{\cal F}_{t+\cdot}\}$, with initial 
distribution 
$\nu (C)=\mathbb{Q}(X(t)\in C)=\mathbb{P}(X(t)\in C|B)$. 
Hence, by the assumptions of the lemma, 
\[E^{\mathbb{P}}[{\bf 1}_Be^t\int_t^{\infty}e^{-\lambda s}h(X(s))
ds]=\mathbb{P}(B)\nu u=E^{\mathbb{P}}[{\bf 1}_Bu(X(t))].\]
Therefore
\[E[\int_0^{\infty}e^{-\lambda s}h(X(s))ds|{\cal F}_t]=e^{-\lambda 
t}u(X(t))+\int_0^te^{-\lambda s}h(X(s))ds,\]
and the assertion follows from Lemma \ref{resid} with 
$f=u$ and $g=\lambda u-h$.
\end{proof}

Our approach to proving uniqueness of the solution of 
the martingale problem for $(A,\mu )$ relies on the following 
theorem.

\begin{definition}\ A class of functions $M\subset {\cal C}_b(E)$ is 
{\em separating\/} if, for $\mu ,\nu\in {\cal P}(E)$, $\mu f=\nu 
f$ for all 
$f\in M$ implies $\mu =\nu$.
\end{definition}

\begin{theorem}\label{lapunq} 
For each $\lambda >0$, suppose that 
\begin{equation}E[\int_0^{\infty}e^{-\lambda s}h(X(s))ds]=E[\int_0^{\infty}e^{-
\lambda s}h(Y(s))ds],\label{lapid}\end{equation}
for any two solutions of the martingale problem for 
$A$ in $D_E[0,\infty )$ with the same initial distribution and  
for all $h$ in a separating class of functions $M_{\lambda}$. 
Then  any two solutions of the martingale problem for 
$A$ in $D_E[0,\infty )$ with the same initial distribution 
have the same distribution on $D_E[0,\infty )$. 
\end{theorem}\ 

\begin{proof}
The proof is actually implicit in the proof of Corollary 
4.4.4 of \cite{EK86}, but we give it here for clarity.
For any two solutions of the martingale problem for 
$A$ in $D_E[0,\infty )$ with the same initial distribution, 
we have, for all $h\in M_{\lambda}$, 
\begin{equation}\int_0^{\infty}e^{-\lambda s}E[h(X(s))]ds=\int_0^{
\infty}e^{-\lambda s}E[h(Y(s))]ds.\label{lamid}\end{equation}
Consider the measures on $(E,{\cal B}(E))$ defined by 
\[m^X(C)=\int_0^{\infty}e^{-\lambda s}E[1_C(X(s))]ds,\quad m^Y(C)
=\int_0^{\infty}e^{-\lambda s}E[1_C(Y(s))]ds.\]
Then $(\ref{lamid})$ reads 
\[m^Xh=m^Yh,\]
and, as $M_{\lambda}$ is separating, (\ref{lamid}) holds for all 
$h\in B(E)$.  Since the identity holds for each $\lambda >0$,
by uniqueness of the Laplace transform, 
\[E[h(X(s))]=E[h(Y(s))],\qquad\mbox{\rm for almost every }s\geq 0
,\]
and by the right continuity of $X$ and $Y$, for all 
$s\geq 0$. Consequently, $X$ and 
$Y$ have the same one-dimensional distributions, and hence 
by Theorem \ref{1dfid}, the same finite-dimensional 
distributions and the same distribution on $D_E[0,\infty )$. 
\end{proof}

\begin{corollary}\label{rngunq}
Suppose that for each $\lambda >0$, 
${\cal R}(\lambda -A)$ is separating.  Then for each initial distribution 
$\mu\in {\cal P}(E)$, any two cadlag solutions of the martingale problem for 
$(A,\mu )$ have the same distribution on $D_E[0,\infty )$.  
\end{corollary}

\begin{proof}
The assertion follows immediately from Lemma \ref{ext} 
and Theorem \ref{lapunq}.
\end{proof}

Martingale problems and dissipative operators are closely 
related.

\begin{definition}
A linear operator $A\subset B(E)\times B(E)$ is {\em dissipative\/} provided
\[\Vert\lambda f-g\Vert\geq\lambda\Vert f\Vert ,\]
for each $(f,g)\in A$ and each $\lambda >0$.
\end{definition}

\begin{lemma}\label{diss}
Suppose that for each $x\in E$, there exists a solution of the 
martingale problem for $(A,\delta_x)$. Then $A$ is dissipative.
\end{lemma}

\begin{proof}
By Lemma \ref{resid},
\[|f(x)|\leq E[\int_0^{\infty}e^{-\lambda s}|\lambda f(X(s))-g(X(
s))|ds]\leq\frac 1{\lambda}\Vert\lambda f-g\Vert .\]
\end{proof}

\subsection{Viscosity solutions}
\label{viscsol}

Let $A\subset C_b(E)\times C_b(E)$.  Theorem \ref{rngunq} implies that 
if for each $\lambda >0$, the equation 
\begin{equation}\lambda u-g=h\label{reso}\end{equation}
has a solution $(u,g)\in A$ for every $h$ in a class 
of functions $M_{\lambda}\subseteq C_b(E)$ that is separating, 
then for each initial distribution $\mu$, the martingale problem for 
$(A,\mu )$ has at most one solution.  
Unfortunately in many situations it is hard to verify 
that $(\ref{reso})$ has a solution in $A$. Thus one is lead to consider 
a weaker notion of solution, namely the notion of 
{\em viscosity solution}. 

\begin{definition}\label{visc}{\bf (viscosity (semi)solution)}

Let $A$ be as above, $\lambda >0$, and $h\in C_b(E)$.

\begin{itemize}

\item[a)]$u\in B(E)$ is a {\em viscosity subsolution\/} of (\ref{reso})
if and only if $u$ is upper semicontinuous  and 
if $(f,g)\in A$ and $x_0\in E$ satisfy
\begin{equation}\sup_x(u-f)(x)=(u-f)(x_0),\label{subtest}\end{equation}
then
\begin{equation}\lambda u(x_0)-g(x_0)\leq h(x_0).\label{subeq}\end{equation}

\item[b)]$u\in B(E)$ is a {\em viscosity supersolution\/} of $(\mbox{\rm \ref{reso}}
)$ 
if and only if $u$ is lower semicontinuous  and 
if $(f,g)\in A$ and $x_0\in E$ satisfy 
\begin{equation}\inf_x(u-f)(x)=(u-f)(x_0),\label{suptest}\end{equation}
then
\begin{equation}\lambda u(x_0)-g(x_0)\geq h(x_0).\label{supeq}\end{equation}
\end{itemize}

A function $u\in C_b(E)$ is a 
{\em viscosity solution\/} of (\ref{reso}) if it is both a 
subsolution and a supersolution.  
\end{definition}

In the theory of viscosity solutions, usually existence of 
a viscosity solution follows by existence of a viscosity 
subsolution and a viscosity supersolution, together with 
the following {\em comparison principle}.

\begin{definition}\label{compar} The {\em comparison principle }
holds for $(\mbox{\rm \ref{reso}})$ when every subsolution is 
pointwise less than or equal to every supersolution.
\end{definition}

\begin{remark}\label{do-in-v}
To better motivate the notion of viscosity solution in the 
context of martingale problems, assume that there exists 
a solution of the martingale problem for $(A,\delta_x)$ for each 
$x\in E$.  Suppose that 
there exists $v\in C_b(E)$ such that 
\begin{equation}e^{-\lambda t}v(X(t))+\int_0^te^{-\lambda s}h(X(s
))ds\label{lammgp}\end{equation}
is a $\{{\cal F}^X_t\}$-martingale for every solution $X$ of the 
martingale problem for $A$.  
Let $(f,g)\in A$ and $x_0$ satisfy
\[\sup_x(v-f)(x)=(v-f)(x_0).\]
Let $X$ be a solution of the martingale problem for $(A,\delta_{x_
0})$. 
Then 
\[e^{-\lambda t}(v(X(t))-f(X(t)))+\int_0^te^{-\lambda s}(h(X(s)-\lambda 
f(X(s))+g(X(s)))ds\]
is a $\{{\cal F}_t^X\}$-martingale by Lemma \ref{resid}, and 
\begin{eqnarray*}
&&E[\int_0^te^{-\lambda s}(\lambda v(X(s))-g(X(s))-h(X(s)))ds]\\
&&\qquad\qquad =E[\int_0^te^{-\lambda s}\lambda (v(X(s))-f(X(s)))
ds]\\
&&\qquad\qquad\qquad\qquad +E[e^{-\lambda t}(v(X(t)-f(X(t)))-(v(x_
0)-f(x_0))]\\
&&\qquad\qquad\leq 0.\end{eqnarray*}
Dividing by $t$ and letting $t\rightarrow 0$, we see that
\[\lambda v(x_0)-g(x_0)\leq h(x_0),\]
so $v$ is a subsolution for (\ref{reso}).  A similar argument 
shows that it is also a supersolution and hence a 
viscosity solution. We will give conditions such  that if the comparison 
principle holds for some $h$, then a viscosity solution $v$ exists and (\ref{lammgp}) is a 
martingale for every solution of the martingale 
problem for $A$. 
\end{remark}\medskip

In the case of a domain with boundary, in order to 
uniquely determine the solution of the martingale problem for $A$ one usually
must specify some boundary conditions, by means of 
boundary operators $B_1\ldots ,B_m$.

Let $E_0\subseteq E$ be an open set and let 
\[\partial E_0=\cup_{k=1}^mE_k,\]
for disjoint, nonempty Borel sets $E_1,\ldots ,E_m$.

Let $A\subseteq C_b(\bar {E}_0)\times C_b(\bar {E}_0)$, $B_k\subseteq 
C_b(\bar {E}_0)\times C(\bar {E}_0)$, $k=1,...,m$, 
be linear operators with a common domain ${\cal D}$. 
For simplicity we will assume that $ $$E$ is compact (hence 
the subscript $b$ will be dropped) and that $A,B_1,\ldots ,B_m$ are 
single valued. 

\begin{definition}
Let $A,B_1\ldots ,B_m$ be as above, and let $\lambda >0$.  For 
$h\in C_b(\bar {E}_0)$, consider the equation 
\begin{eqnarray}
\lambda u-Au&=&h,\quad\mbox{\rm on }E_0\label{creso}\\
-B_ku&=&0,\quad\mbox{\rm on }E_k,\quad k=1,\cdots ,m.\nonumber
\end{eqnarray}

\begin{itemize}
\item[a)]$u\in B(\bar {E}_0)$ is a {\em viscosity subsolution\/} of (\ref{creso})
if and only if $u$ is upper semicontinuous  and 
if $f\in {\cal D}$ and $x_0\in\bar {E}_0$ satisfy 
\begin{equation}\sup_x(u-f)(x)=(u-f)(x_0),\label{csubtest}\end{equation}
then
\begin{eqnarray}
\lambda u(x_0)-Af(x_0)\leq h(x_0),&\qquad&\mbox{\rm if }x_0\in E_
0,\label{interior}\\
(\lambda u(x_0)-Af(x_0)-h(x_0))\wedge\min_{k:x_0\in\bar {E}_k}(-B_
kf(x_0))\leq 0,&\qquad&\mbox{\rm if }x_0\in\partial E_0.\label{bndry}
\end{eqnarray}

\item[b)]$u\in B(\bar {E}_0)$ is a {\em viscosity supersolution\/} of 
$(\mbox{\rm \ref{creso}})$ 
if and only if $u$ is lower semicontinuous  and 
if $f\in {\cal D}$ and $x_0\in\bar {E}_0$ satisfy 
\begin{equation}\inf_x(u-f)(x)=(u-f)(x_0),\label{csuptest}\end{equation}
then
\begin{eqnarray}
\lambda u(x_0)-Af(x_0)\geq h(x_0),&\qquad&\mbox{\rm if }x_0\in E_
0,\label{csupeq}\\
(\lambda u(x_0)-Af(x_0)-h(x_0))\vee\max_{k:x_0\in\bar {E}_k}(-B_k
f(x_0))\geq 0,&\qquad&\mbox{\rm if }x_0\in\partial E_0.\nonumber\end{eqnarray}
\end{itemize}

A function $u\in C(\bar {E}_0)$ is a 
{\em viscosity solution\/} of (\ref{creso}) if it is both a 
subsolution and a supersolution.  

\end{definition}

\begin{remark}
The above definition, with the `relaxed' requirement that 
on the boundary either the interior inequality or the 
boundary inequality be satisfied by at least one among 
$-B_1f,\cdots ,-B_mf$ is the standard one in the theory of 
viscosity solutions 
where it is used in particular because it is stable under 
limit operations and because it can be localized. As will be 
clear in Section \ref{cmgpunq}, it suits perfectly our 
approach to martingale problems with boundary conditions.
\end{remark}\medskip

\section{Comparison principle and uniqueness for 
martingale problems}\label{mgpunq}

In this section, we restrict our attention to  
\[A\subset C_b(E)\times C_b(E)\]
and consider the martingale problem for $A$ in $D_E[0,\infty )$.  
Let $\Pi\subset {\cal P}(D_E[0,\infty ))$ denote the collection of distributions 
of solutions of 
the martingale problem for $A$ in $D_E[0,\infty )$, 
and, for $\mu\in {\cal P}(E)$, let 
$\Pi_{\mu}\subset\Pi$ denote the subcollection with initial distribution 
$\mu$.  If $\mu =\delta_x$, we will write $\Pi_x$ for $\Pi_{\delta_
x}$. In this section  
$X$ will be the canonical process on $D_E[0,\infty )$. 
Assume the following condition.

\begin{condition}\label{H}\hfill

\begin{itemize}
\item[a)] ${\cal D}(A)$ is dense in $C_b(E)$ in the topology of 
uniform convergence on compact sets.

\item[b)] For each $\mu\in {\cal P}(E)$, $\Pi_{\mu}\neq\emptyset$.

\item[c)]
If ${\cal K}\subset {\cal P}(E)$ is compact, then $\cup_{\mu\in {\cal K}}
\Pi_{\mu}$ is compact.  (See 
Proposition \ref{comp} below.)
\end{itemize}
\end{condition}

\begin{remark}\label{Jak}
In working with these conditions, it is simplest to take 
the usual Skorohod topology on $D_E[0,\infty )$. (See, for 
example, Sections 3.5-3.9 of \cite{EK86}.) The results of 
this paper also hold if we take the Jakubowski topology 
 (\cite{Jak97}).  The $\sigma$-algebra of 
Borel sets ${\cal B}(D_E[0,\infty ))$ is 
the same for both topologies and, in fact, is simply the 
smallest $\sigma$-algebra under which all mappings of the form 
$x\in D_E[0,\infty )\rightarrow x(t)$, $t\geq 0$, are measurable. 

It is also relevant to note that mappings of the form
\[x\in D_E[0,\infty )\rightarrow\int_0^{\infty}e^{-\lambda t}h(x(
t))dt,\quad h\in C_b(E),\lambda >0,\]
are continuous under both topologies.  The Jakubowski 
topology could be particularly useful for extensions of 
the results of Section \ref{cmgpunq} to constrained 
martingale problems in which the boundary terms are 
not local-time integrals.
\end{remark}\medskip

\begin{proposition}\label{comp}In addition to Condition 
\ref{H}(a), assume that
for each compact $K\subset E$, $\varepsilon >0$, and $T>0$, there 
exists a compact $K'\subset E$ such that 
\[P\{X(t)\in K',t\leq T,X(0)\in K\}\geq (1-\varepsilon )P\{X(0)\in 
K\},\quad\forall P\in\Pi .\]
Then Condition \ref{H}(c) holds.
\end{proposition}

\begin{proof}
The assertion is part of the thesis of Theorem 4.5.11 (b) of \cite{EK86}.
\end{proof}

Let $\lambda >0$, and for $h\in C_b(E)$, define 
\begin{equation}u_{+}(x)=u_{+}(x,h)=\sup_{P\in\Pi_x}E^P[\int_0^{\infty}
e^{-\lambda t}h(X(t))dt],\label{uplus}\end{equation}
\begin{equation}\quad u_{-}(x)=u_{-}(x,h)=\inf_{P\in\Pi_x}E^P[\int_
0^{\infty}e^{-\lambda t}h(X(t))dt].\label{uminus}\end{equation}
\begin{equation}\pi_{+}(\Pi_{\mu},h)=\sup_{P\in\Pi_{\mu}}E^P[\int_
0^{\infty}e^{-\lambda t}h(X(t))dt],\label{piplus}\end{equation}
\begin{equation}\pi_{-}(\Pi_{\mu},h)=\inf_{P\in\Pi_{\mu}}E^P[\int_
0^{\infty}e^{-\lambda t}h(X(t))dt].\label{piminus}\end{equation}

\begin{lemma}\label{intrep}
Under Condition \ref{H}, for $h\in C_b(E)$, 
$u_{+}(x,h)$ is upper semicontinuous (hence measurable), and 
\begin{equation}\pi_{+}(\Pi_{\mu},h)=\int_Eu_{+}(x,h)\mu (dx)\qquad
\forall\mu\in {\cal P}(E).\label{intrep2}\end{equation}
The analogous result holds for $u_{-}$ and $\pi_{-}$.  
\end{lemma}

\begin{proof}
For $\pi_{+}$, the lemma is a combination of Theorem 4.5.11(a), Lemma 4.5.8, 
Lemma 4.5.9 and Lemma 4.5.10 of \cite{EK86}, but we recall here 
the main steps of the proof for the convenience of the 
reader. Throughout the proof, $h$ will be fixed and 
will be omitted. In addition we will use the notation $\pi_{+}(\Pi_{
\mu},h)=\pi_{+}(\mu )$. 

First of all let us show that, for $\mu_n\rightarrow\mu$, 
\[\limsup_{n\rightarrow \infty}\pi_{+}(\mu_n)\leq\pi_{+}(\mu ).\]
In fact, by the compactness of $\Pi_{\mu}$ there is $P\in\Pi_{\mu}$ that 
achieves the supremum. Moreover, by  
the compactness of $\Pi_{\mu}\cup\cup_n\Pi_{\mu_{n_i}}$, for every convergent 
subsequence $\{\pi_{+}(\mu_{n_i})\}=\left\{E^{P_{n_i}}[\int_0^{\infty}
e^{-\lambda t}h(X(t))dt]\right\}$, we can extract a 
subsequence $\left\{P_{n_{i_j}}\right\}$ that converges to some $
P\in\Pi_{\mu}$. 
Since $\int_0^{\infty}e^{-\lambda t}h(x(t))dt]$ is continuous on $
D_E[0,\infty )$, we then have  
\[\lim_{i\rightarrow\infty}\,\pi_{+}(\mu_{n_i})=\lim_{j\rightarrow\infty}\,E^{P_{n_{i_j}}}[\int_0^{\infty}
e^{-\lambda t}h(X(t))dt]=E^P[\int_0^{\infty}e^{-\lambda t}h(X(t))
dt]\leq\pi_{+}(\mu ).\]

This yields, in particular, the upper semicontinuity (and 
hence the measurability) of $u_{+}(x)=\pi_{+}(\delta_x)$. 

Next, Condition \ref{H} (b) and (c) implies that, for $\mu_1,\mu_
2\in {\cal P}(E)$, 
$0\leq\alpha\leq 1$, 
\[\pi_{+}(\alpha\mu_1+(1-\alpha )\mu_2)=\alpha\pi_{+}(\mu_1)+(1-\alpha 
)\pi_{+}(\mu_2),\]
(Theorem 4.5.11(a), Lemma 4.5.8 and Lemma 4.5.10 of 
\cite{EK86}. We will not recall this part of the 
proof). This yields, for $\{\mu_i\}\subset {\cal P}(E)$, $\alpha_
i\geq 0$, $\sum_i\alpha_i=1$: 
\[\pi_{+}(\sum_i\alpha_i\mu_i)=\sum_1^N\alpha_i\,\pi_{+}(\frac 1{
\sum_1^N\alpha_i}\sum_1^N\alpha_i\mu_i)+\sum_{N+1}^{\infty}\alpha_
i\,\pi_{+}(\frac 1{\sum_{N+1}^{\infty}\alpha_i}\sum_{N+1}^{\infty}
\alpha_i\mu_i),\]
and hence 
\[\Big|\pi_{+}(\sum_i\alpha_i\mu_i)-\sum_i\alpha_i\pi_{+}(\mu_i)\Big
|\leq 2\frac {||h||}{\lambda}\sum_{N+1}^{\infty}\alpha_i,\qquad\forall 
N.\]

Finally, for each $n$, let $\{E^n_i\}$ be a countable collection of disjoint 
subsets of $E$ with diameter less than $\frac 1n$ and such that 
$E=\bigcup_iE^n_i$. In addition let $x^n_i\in E^n_i$ satisfy 
$ $$u_{+}(x^n_i)\geq\sup_{E^n_i}u_{+}(x)-\frac 1n$. Define $u_n(x
)=\sum_iu_{+}(x^n_i)1_{E^n_i}(x)$ 
and $\mu_n=\sum_i\mu (E^n_i)\delta_{x^n_i}.$ 
Then $\{u_n\}$ converges to $u_{+}$ pointwise and boundedly, and 
$\{\mu_n\}$ converges to $\mu$. Therefore 
\[\int_Eu_{+}(x)\mu (dx)=\lim_n\int_Eu_n(x)\mu (dx)=\lim_n\sum_i\pi_{
+}(\delta_{x^n_i})\mu (E^n_i)=\lim_n\,\pi_{+}(\mu_n)\leq\pi_{+}(\mu 
).\]
To prove the opposite inequality, let 
$\mu_i^n(B)=\mu (B\cap E^n_i)/\mu (E^n_i)$, for $\mu (E^n_i)>0$, and 
$u_n(x)=\sum_i\pi_{+}(\mu^n_i)1_{E^n_i}(x)$. For each $x\in E$, for every $
n$ 
there exists a (unique) $i(n)$ such that $x\in E_{i(n)}^n$. Then 
$u_n(x)=\pi_{+}(\mu^n_{i(n)})$ and $\mu^n_{i(n)}\rightarrow\delta_
x$, hence $\limsup_nu_n(x)\leq u_{+}(x).$ 
Therefore 
\[\pi_{+}(\mu )=\pi_{+}(\sum_i\mu (E^n_i)\,\mu_i^n)=\int_Eu_n(x)\mu 
(dx)\leq\int_E\limsup_n\,u_n(x)\,\mu (dx)\leq\int_Eu_{+}(x)\mu (d
x),\]
where the last but one inequality follows from the fact 
that the $u_n$ are uniformly bounded. 

To prove the assertion 
for $\pi_{-}$ use the fact that $\pi_{-}(\Pi_{\mu},h)=-\pi_{+}(\Pi_{
\mu},-h)$.
\end{proof}

\begin{lemma}\label{subsol}Assume 
that Condition \ref{H} holds.  Then $u_{+}$ is a viscosity 
subsolution of (\ref{reso}) and $u_{-}$ is a viscosity 
supersolution of the same equation.  
\end{lemma}

\begin{proof}
Since 
$u_{-}(x,h)=-u_{+}(x,-h)$ it is enough to consider $u_{+}$.  Let 
$(f,g)\in A$.  Suppose $x_0$ is a point such that 
$u_{+}(x_0)-f(x_0)=\sup_x(u_{+}(x)-f(x))$.  Since we can always 
add a constant to $f$, we can assume $u_{+}(x_0)-f(x_0)=0$.  
By compactness (Condition \ref{H}(c)), we have 
\[u_{+}(x_0)=E^P\left[\int_0^{\infty}e^{-\lambda t}h(X(t)dt\right
]\]
for some $P\in\Pi_{x_0}$.

For $\epsilon >0$, define 
\begin{equation}\tau_{\epsilon}=\epsilon\wedge\inf\{t>0:r(X(t),x_0)\geq\epsilon\mbox{\rm \ or }
r(X(t-),x_0)\geq\epsilon \}\label{tau}\end{equation}
and let $H_{\epsilon}=e^{-\lambda\tau_{\epsilon}}$. 
Then, by Lemma \ref{resid},
\begin{eqnarray*}
0&&=u_{+}(x_0)-f(x_0)\\
&&=E^P\left[\int_0^{\infty}e^{-\lambda t}\left(h(X(t))-\lambda f(
X(t))+g(X(t))\right)dt\right]\\
&&=E^P\left[\int_0^{\tau_{\epsilon}}e^{-\lambda t}\left(h(X(t))-\lambda 
f(X(t))+g(X(t))\right)dt\right]\\
&&\qquad\qquad\quad +E^P\left[e^{-\lambda\tau_{\epsilon}}\int_0^{
\infty}e^{-\lambda t}\left(h(X(t+\tau_{\epsilon}))-\lambda f(X(t+
\tau_{\epsilon}))+g(X(t+\tau_{\epsilon}))\right)dt\right]\\
&&=E^P\left[\int_0^{\tau_{\epsilon}}e^{-\lambda t}\left(h(X(t))-\lambda 
f(X(t))+g(X(t))\right)dt\right]\\
&&\qquad\quad\qquad +E^P[H_{\epsilon}]E^{P^{\tau_{\epsilon},H_{\epsilon}}}\left
[\int_0^{\infty}e^{-\lambda t}\left(h(X(t))-\lambda f(X(t))+g(X(t
))\right)dt\right].\\
\end{eqnarray*}
Setting $\mu_{\epsilon}(\cdot )=P^{\tau_{\epsilon},H_{\epsilon}}(
X(0)\in\cdot )$, by Lemma 
\ref{rnlem} and Lemma \ref{resid}, the above chain of 
equalities can be continued as (with the notation 
$(\mbox{\rm \ref{notat-i}})$) 
\[\begin{array}{rcl}
&&\leq E^P\left[\int_0^{\tau_{\epsilon}}e^{-\lambda t}\left(h(X(t
))-\lambda f(X(t))+g(X(t))\right)dt\right]+E^P[H_{\epsilon}](\pi_{
+}(\Pi_{\mu_{\epsilon}},h)-\mu_{\epsilon}f)\end{array}
,\]
and, by Lemma \ref{intrep}, 
\begin{eqnarray*}
&&=E^P\left[\int_0^{\tau_{\epsilon}}e^{-\lambda t}\left(h(X(t))-\lambda 
f(X(t))+g(X(t))\right)dt\right]+E^P[H_{\epsilon}](\mu_{\epsilon}u_{
+}-\mu_{\epsilon}f)\\
&&=E^P\left[\int_0^{\tau_{\epsilon}}e^{-\lambda t}\left(h(X(t))-\lambda 
f(X(t))+g(X(t))\right)dt\right]+E^P\left[e^{-\lambda\tau_{\epsilon}}
(u_{+}(X(\tau_{\epsilon}))-f(X(\tau_{\epsilon}))\right]\\
&&\leq E^P\left[\int_0^{\tau_{\epsilon}}e^{-\lambda t}\left(h(X(t
))-\lambda f(X(t))+g(X(t))\right)dt\right],\end{eqnarray*}
where the last inequality uses the fact that 
$u_{+}-f\leq 0$. Therefore 
\begin{eqnarray*}
0&\leq&\lim_{\epsilon\rightarrow 0}\frac {E^P\left[\int_0^{\tau_{
\epsilon}}e^{-\lambda t}\left(h(X(t))-\lambda f(X(t))+g(X(t))\right
)dt\right]}{E^P[\tau_{\epsilon}]}\\
&=&h(x_0)-\lambda f(x_0)+g(x_0)\\
&=&h(x_0)-\lambda u_{+}(x_0)+g(x_0).\end{eqnarray*}
\end{proof}

\begin{corollary}\label{sol} 
Let $h\in C_b(E)$.  If in addition to Condition \ref{H},
the comparison principle holds for 
equation (\ref{reso}), then 
$u=u_{+}=u_{-}$ is the unique viscosity solution of equation 
(\ref{reso}).  
\end{corollary}

\begin{theorem}\label{mgpb} 
Assume that Condition \ref{H} holds.  For $\lambda >0$, let $M_{\lambda}$ 
be the set of $h\in C_b(E)$, such that the comparison 
principle holds for (\ref{reso}).  If for each $\lambda >0$, 
$M_{\lambda}$ is separating, then uniqueness holds for the martingale 
problem for $A$ in $D_E[0,\infty )$.  
\end{theorem}

\begin{proof}
If the comparison principle for (\ref{reso}) holds for some 
$h\in C_b(E)$, then by Lemma \ref{subsol}, $u_{+}=u_{-}$.  Then, 
by the definition of $u_{+}$ and $u_{-}$ and Lemma \ref{intrep}, for any two solutions  
$P_1,P_2\in\Pi_{\mu}$, we must have
\[E^{P_1}[\int_0^{\infty}e^{-\lambda t}h(X(t))dt]=E^{P_2}[\int_0^{
\infty}e^{-\lambda t}h(X(t))dt].\]
Consequently, if $M_{\lambda}$ is separating, Theorem \ref{lapunq} 
implies $P_1=P_2$.
\end{proof}

\begin{remark}\label{extunq}
Another way of viewing the role of the comparison 
principle in the issue of uniqueness for the martingale 
problem for $A$ is the following. 

Suppose the comparison principle holds for some $h$ and 
let $u_{+}=u_{-}=u$. By Lemmas \ref{intrep} and \ref{ext}, 
\[u(X(t))-\int_0^t(\lambda u(X(s))-h(X(s)))ds,\]
is a $\left\{{\cal F}_t^X\right\}$-martingale for every $P\in\Pi$. 
Then, even though 
$u_{+}$ and $u_{-}$ are defined in nonlinear ways, the linearity of 
the martingale property ensures  
that $A$ can be extended to the linear span $A^u$ of 
$A\cup \{(u,\lambda u-h)\}$ and every solution of the martingale 
problem for $A$ will be a solution of the martingale 
problem for $A^u$. By applying this procedure to 
all functions $h\in M_{\lambda}$, we obtain an extension $\hat {A}$ of $
A$ such 
that every solution of the martingale 
problem for $A$ will be a solution of the martingale 
problem for $\hat {A}$ and such that ${\cal R}(\lambda -\hat {A})
\supset M_{\lambda}$ and hence is 
separating. Therefore uniqueness follows from Corollary 
\ref{rngunq}.

Notice that, even if the comparison 
principle does not hold, under Condition \ref{H}, by 
Lemma 4.5.18 of \cite{EK86}, for each $\mu\in {\cal P}(E)$, there 
exists $P\in\Pi_{\mu}$ such that under $P$ 
\[u_{+}(X(t))-\int_0^t(\lambda u_{+}(X(s))-h(X(s)))ds\]
is a $\{{\cal F}_t^X\}$-martingale.
\end{remark}\medskip

\begin{remark}
If, for some $h$, there exists $(u,g)\in\hat {A}$ such that $\lambda 
u-g=h$ 
(essentially $u$ is the analog of a stochastic solution as 
defined in \cite{SV72b}), then, by Lemma \ref{resid} and 
Remark \ref{do-in-v}, $u$ is a viscosity solution of $(\mbox{\rm \ref{reso}}
)$.
\end{remark}\medskip

\section{Alternative definitions of viscosity solution} \label{altdef}
Different definitions of viscosity solution may be useful, 
depending on the setting.  Here we discuss two other 
possibilities.  As mentioned in the Introduction, the first, 
which is stated in terms of 
solutions of the martingale problem, is a modification of 
a definition used in 
\cite{EKTZ14}, while the second is a stronger version 
of definitions in \cite{FK06}. 
We show that Lemma \ref{subsol} still holds under these 
alternative definitions and hence all the results of Section 
\ref{mgpunq} carry over. Both definitions are stronger in 
the sense that the inequalities (\ref{subeq}) and 
(\ref{supeq}) required by the previous definition are 
required by these definitions. Consequently, in both 
cases, it should be easier to prove the comparison 
principle.

${\cal T}$ will denote the set of $\{{\cal F}_t^X\}$-stopping 
times.  

\begin{definition}\label{visc2}{\bf (Stopped viscosity (semi)solution)}

Let $A\subset C_b(E)\times C_b(E)$, $\lambda >0$, and $h\in C_b(E
)$.

\begin{itemize}

\item[a)]$u\in B(E)$ is a {\em stopped viscosity subsolution\/} of (\ref{reso})
if and only if $u$ is upper semicontinuous  and 
if $(f,g)\in A$, $x_0\in E$, and there exists 
a strictly positive $\tau_0\in {\cal T}$ such that
\begin{equation}\sup_{P\in\Pi_{x_0},\,\tau\in {\cal T}}\frac {E^P
[e^{-\lambda\tau\wedge\tau_0}(u-f)(X(\tau\wedge\tau_0))]}{E^P[e^{
-\lambda\tau\wedge\tau_0}]}=(u-f)(x_0),\label{subtest2}\end{equation}
 then 
\begin{equation}\lambda u(x_0)-g(x_0)\leq h(x_0).\label{subeq2}\end{equation}

\item[b)]$u\in B(E)$ is a {\em stopped viscosity supersolution\/} of (\ref{reso})
if and only if $u$ is lower semicontinuous  and 
if $(f,g)\in A$,  $x_0\in E$, and there exists 
a strictly positive $\tau_0\in {\cal T}$ such that
\begin{equation}\inf_{P\in\Pi_{x_0},\,\tau\in {\cal T}}\frac {E^P
[e^{-\lambda\tau\wedge\tau_0}(u-f)(X(\tau\wedge\tau_0))]}{E^P[e^{
-\lambda\tau\wedge\tau_0}]}=(u-f)(x_0),\label{suptest2}\end{equation}
then
\begin{equation}\lambda u(x_0)-g(x_0)\geq h(x_0).\label{supeq2}\end{equation}

\end{itemize}

A function $u\in C_b(E)$ is a 
{\em stopped viscosity solution\/} of (\ref{reso}) if it is both a 
subsolution and a supersolution.  

\end{definition}

\begin{remark}
If $(u-f)(x_0)=\sup_x (u-f)(x)$, then (\ref{subtest2}) is satisfied. Consequently,
every stopped sub/supersolution in the sense of Definition \ref{visc2} is a 
sub/supersolution in the sense of Definition \ref{visc}.
\end{remark}\medskip

\begin{remark}
Definition \ref{visc2} requires $(\ref{subeq2})$ 
($(\ref{supeq2})$) to hold only at points $x_0$ for which 
$(\ref{subtest2})$ ($(\ref{suptest2})$) is verified for 
some $\tau_0$. Note that, as in Definition \ref{visc}, such an $x_
0$ might not exist. 

Definition \ref{visc2} essentially requires a local maximum 
principle and is  related to the notion of
characteristic operator as given in \cite{Dyn65}.
\end{remark}\medskip

For Definition \ref{visc2}, we have the following analog 
of Lemma \ref{subsol}.

\begin{lemma}\label{subsol2}Assume 
that Condition \ref{H} holds.  Then 
$u_{+}$ given by (\ref{uplus}) is a stopped viscosity 
subsolution of (\ref{reso}) and $u_{-}$ given by (\ref{uminus})
is a stopped viscosity 
supersolution of the same equation.  
\end{lemma}

\begin{proof}
Let $(f,g)\in A$.  Suppose $x_0$ is a point such that 
$(\ref{subtest2})$ holds for $u_{+}$ for some 
$\tau_0\in {\cal T}$, $\tau_0>0$. 
Since we can always 
add a constant to $f$, we can assume $u_{+}(x_0)-f(x_0)=0$.  
By the same arguments used in the proof of Lemma 
\ref{subsol}, defining $\tau_{\epsilon}$ and $H_{\epsilon}$ in the same way, we 
obtain, for some $P\in\Pi_{x_0}$(independent of $\epsilon )$, 
\begin{eqnarray*}
0&&=u_{+}(x_0)-f(x_0)\\
&&\leq E^P\left[\int_0^{\tau_{\epsilon}\wedge\tau_0}e^{-\lambda t}\left
(h(X(t))-\lambda f(X(t))+g(X(t))\right)dt\right]\\
&&\qquad\qquad\quad +E^P\left[e^{-\lambda\tau_{\epsilon}\wedge\tau_
0}(u_{+}(X(\tau_{\epsilon}\wedge\tau_0))-f(X(\tau_{\epsilon}\wedge
\tau_0))\right]\\
&&\leq E^P\left[\int_0^{\tau_{\epsilon}\wedge\tau_0}e^{-\lambda t}\left
(h(X(t))-\lambda f(X(t))+g(X(t))\right)dt\right],\end{eqnarray*}
where the last inequality uses $(\ref{subtest2})$ 
and the fact that 
$u_{+}(x_0)-f(x_0)=0$.  Then the result follows as in 
Lemma \ref{subsol}.
\end{proof}

The following is essentially Definition 7.1 of \cite{FK06}.

\begin{definition}\label{visc3}{\bf (Sequential viscosity (semi)solution)}

Let $A\subset C_b(E)\times C_b(E)$, $\lambda >0$, and $h\in C_b(E
)$.

\begin{itemize}

\item[a)]$u\in B(E)$ is a {\em sequential viscosity subsolution }
of (\ref{reso})
if and only if $u$ is upper semicontinuous  and 
for each $(f,g)\in A$ and each sequence $y_n\in E$ satisfying
\begin{equation}\lim_{n\rightarrow\infty}(u-f)(y_n)=\sup_x(u-f)(x
),\label{subtest3}\end{equation}
we have
\begin{equation}\limsup_{n\rightarrow\infty}(\lambda u(y_n)-g(y_n
)-h(y_n))\leq 0.\label{subeq3}\end{equation}

\item[b)]$u\in B(E)$ is a {\em sequential viscosity supersolution\/} of (\ref{reso})
 
if and only if $u$ is lower semicontinuous  and 
for each $(f,g)\in A$ and each sequence $y_n\in E$ satisfying
\begin{equation}\lim_{n\rightarrow\infty}(u-f)(y_n)=\inf_x(u-f)(x
),\label{suptest3}\end{equation}
we have
\begin{equation}\liminf_{n\rightarrow\infty}(\lambda u(y_n)-g(y_n
)-h(y_n))\geq 0.\label{supeq3}\end{equation}
\end{itemize}

A function $u\in C_b(E)$ is a 
{\em sequential viscosity solution\/} of (\ref{reso}) if it is both a 
subsolution and a supersolution.  
\end{definition}

\begin{remark}\label{3vs1}
For $E$ compact, every viscosity sub/supersolution is a  
sequential viscosity sub/supersolution.
\end{remark}\medskip

For sequential viscosity semisolutions, we have the following analog 
of Lemma \ref{subsol}.  $C_{b,u}(E)$ denotes the space of 
bounded, uniformly continuous functions on $E$.

\begin{lemma}\label{subsol3}
For $\epsilon >0$, define 
\[\tau_{\epsilon}=\epsilon\wedge\inf\{t>0:r(X(t),X(0))\geq\epsilon\mbox{\rm \ or }
r(X(t-),X(0))\geq\epsilon \}.\]
Assume $A\subset C_{b,u}(E)\times C_{b,u}(E)$,  for each $\epsilon 
>0$, 
$\inf_{P\in\Pi}E^P[\tau_{\epsilon}]>0$, and 
that Condition \ref{H} holds.  Then, for $h\in C_{b,u}(E)$,
$u_{+}$ given by $(\ref{uplus})$ is a sequential viscosity 
subsolution of (\ref{reso}) and $u_{-}$ given by $(\ref{uminus})$
is a sequential viscosity 
supersolution of the same equation.  
\end{lemma}

\begin{proof}
 Let $(f,g)\in A$.  Suppose $\{y_n\}$ is a sequence such that 
$(\ref{subtest3})$ holds for $u_{+}$. 
Since we can always 
add a constant to $f$, we can assume $\sup_x(u_{+}-f)(x)=0$.  
Let $H_{\epsilon}=e^{-\lambda\tau_{\epsilon}}$. Then, by the same arguments as in 
Lemma \ref{subsol}, we have, for some 
$P_n\in\Pi_{y_n}$(independent of $\epsilon )$, 
\begin{eqnarray*}
(u_{+}-f)(y_n)\leq E^{P_n}\left[\int_0^{\tau_{\epsilon}}e^{-\lambda 
t}\left(h(X(t))-\lambda f(X(t))+g(X(t))\right)dt\right],\end{eqnarray*}
where we have used the fact that 
$\sup_x(u_{+}-f)(x)=0$. Therefore 
\[\frac {(u_{+}-f)(y_n)}{E^{P_n}[\tau_{\epsilon}]}\leq\frac {E^{P_
n}\left[\int_0^{\tau_{\epsilon}}e^{-\lambda t}\left(h(X(t))-\lambda 
f(X(t))+g(X(t))\right)dt\right]}{E^{P_n}[\tau_{\epsilon}]}.\]
Replacing $\epsilon$ by $\epsilon_n$ going to zero sufficiently slowly so 
that the left side converges to zero, the uniform 
continuity of $f$, $g$, and $h$ implies the right side is 
asymptotic to $h(y_n)-\lambda f(y_n)+g(y_n)$ giving
\begin{eqnarray*}
0&\leq&\liminf_{n\rightarrow\infty}(h(y_n)-\lambda f(y_n)+g(y_n))\\
&=&\liminf(h(y_n)-\lambda u_{+}(y_n)+g(y_n).\end{eqnarray*}
\end{proof}

The following theorem is essentially Lemma 7.4 of 
\cite{FK06}.  It gives the intuitively natural result that 
if $h\in\overline {{\cal R}(\lambda -A)}$ (where the closure is taken under uniform 
convergence), 
then the comparison principle 
holds for sequential viscosity semisolutions of $\lambda u-Au=h$. 

If $E$ is compact, the same results hold for viscosity 
semisolutions, by Remark \ref{3vs1}. 

\begin{theorem}\label{converse}
Suppose $h\in C_b(E)$ and there 
exist $(f_n,g_n)\in A$ satisfying $\sup_x|\lambda f_n(x)-g_n(x)-h
|\rightarrow 0$.
Then the comparison principle holds for sequential 
viscosity semisolutions of $(\ref{reso})$. 
\end{theorem}

\begin{proof}
Suppose $\underline u$ is a sequential viscosity subsolution.
Set $h_n=\lambda f_n-g_n$.  
For $\epsilon_n>0$, $\epsilon_n\rightarrow 0$, there exist $y_n\in 
E$ 
satisfying 
$\underline u(y_n)-f_n(y_n)\geq\sup_x(\underline u(x)-f_n(x))-\epsilon_
n/\lambda$ and 
$\lambda\underline u(y_n)-g_n(y_n)-h(y_n)\leq\epsilon_n$.  
Then
\begin{eqnarray*}
\sup_x(\lambda\underline u(x)-\lambda f_n(x))&\leq&\lambda\underline 
u(y_n)-\lambda f_n(y_n)+\epsilon_n\\
&\leq&h(y_n)+g_n(y_n)-\lambda f_n(y_n)+2\epsilon_n\\
&=&h(y_n)-h_n(y_n)+2\epsilon_n\\
&\rightarrow&0.\end{eqnarray*}
Similarly, if $\overline u$ is a supersolution of $\lambda u-Au=h$, 
\[\liminf_{n\rightarrow\infty}\inf_x(\overline u(x)-f_n(x))\geq 0
,\]
and it follows that $\underline u\leq\overline u$.
\end{proof}

\section{Martingale problems with boundary 
conditions} \label{cmgpunq}

The study of stochastic processes that are constrained 
to some set $\bar {E}_0$ and must satisfy some boundary 
condition on $\partial E_0$, described by one or more boundary 
operators $B_1,\ldots ,B_m$, is typically carried out by 
incorporating the boundary condition in the definition of 
the domain ${\cal D}(A)$ (see Remark \ref{0bdry} below).  
However, this approach restricts the problems that can 
be dealt with to fairly regular ones, so we follow the 
formulation of a constrained martingale problem given in 
\cite{Kur90}.  (See also \cite{Kur91,KS01}). 

Let $E_0\subseteq E$ be an open set and let 
\[\partial E_0=\cup_{k=1}^mE_k,\]
for disjoint, nonempty Borel sets $E_1,\ldots ,E_m$.  Let 
$A\subseteq C_b(\bar {E}_0)\times C_b(\bar {E}_0)$, $B_k\subseteq 
C_b(\bar {E}_0)\times C_b(\bar {E}_0)$, $k=1,...,m$, be 
linear operators with a common domain ${\cal D}$ such that 
$(1,0)\in A$, $(1,0)\in B_k$, $k=1,...,m$.  For simplicity we will 
assume that $ $$E$ is compact (hence the subscript $b$ will be 
dropped) and that $A,B_1,\ldots ,B_m$ are single-valued.

\begin{definition}\label{cmgp} 
A stochastic process $X$ with sample paths in $D_{\bar {E}_0}[0,\infty 
)$ is 
a solution of the {\em constrained martingale problem\/} for 
$(A,E_0;B_1,E_1;\ldots ;B_m,E_m)$ provided there exist a filtration 
$\{{\cal F}_t\}$ and continuous, nondecreasing processes $\gamma_
1,\ldots ,\gamma_m$ 
such that $X$, $\gamma_1,\ldots ,\gamma_m$ are $\{{\cal F}_t\}$-adapted,
\[\gamma_k(t)=\int_0^t{\bf 1}_{\bar {E}_k}(X(s-))d\gamma_k(s),\]
and for each $f\in {\cal D}$,
\begin{equation}M_f(t)=f(X(t))-f(X(0))-\int_0^tAf(X(s))ds-\sum_{k
=1}^m\int_0^tB_kf(X(s-))d\gamma_k(s)\label{cmg}\end{equation}
is a $\{{\cal F}_t\}$-martingale. 
\end{definition}

\begin{remark}\label{adap}
$\gamma_1,\ldots ,\gamma_m$ will be called {\em local times\/} since $
\gamma_k$ increases
only when $X$ is in $\bar {E}_k$.  Without loss of 
generality, we can assume that the $\gamma_k$ are $\{{\cal F}^X_t
\}$-adapted. 
(Replace $\gamma_k$ by its dual, predictable projection on $\{{\cal F}_
t^X\}$.)
Definition \ref{cmgp}\ does not require
that the $\gamma_k$ be uniquely determined 
by the distribution of $X$, but if $\gamma_k^1$ and $\gamma_k^2$, $
k=1,\ldots ,m$, are 
continuous and satisfy 
the martingale requirement with the same filtration, 
we must have
\[\sum_{k=1}^m\int_0^tB_kf(X(s-))d\gamma_k^1(s)-\sum_{k=1}^m\int_
0^tB_kf(X(s-))d\gamma_k^2(s)=0,\]
since this expression will be a continuous martingale with 
finite variation paths.
\end{remark}\medskip

\begin{remark}
The main example of a constrained martingale 
problem in the sense of the above definition is the 
constrained martingale problem that describes a reflected 
diffusion process.  In this case, $A$ is a second order elliptic 
operator and the $B_k$ are first order 
differential operators. Although there is a vast 
literature on this topic, there are still relevant cases of 
reflected diffusions that have not been uniquely characterized 
as solutions of martingale problems or stochastic 
differential equations. In Section \ref{degref}, the results 
of this section are used in one of these cases. More 
general constrained diffusions where the $B_k$ are second 
order elliptic operators, for instance diffusions with sticky 
reflection, also satisfy Definition \ref{cmgp}.

Definition \ref{cmgp}\ is a special case of a more general 
definition of constrained martingale problem given in 
\cite{KS01}.  This broader definition allows for more 
general boundary behavior, such as models considered in 
\cite{CK06}.  
\end{remark}\medskip

Many results for solutions of martingale problems carry 
over to solutions of constrained martingale problems. In 
particular Lemma \ref{resid} still holds. 
In addition the following lemma holds. 

\begin{lemma}\label{cres} 
Let $X$ be a stochastic process with sample paths in 
$D_{\bar {E}_0}[0,\infty )$, $\gamma_1,\ldots ,\gamma_m$ be continuous, nondecreasing processes  
such that $X$, $\gamma_1,\ldots ,\gamma_m$ are $\{{\cal F}_t\}$-adapted. Then for 
$f\in {\cal D}$ such that $(\ref{cmg})$ is a $\{{\cal F}_t\}$-martingale and $
\lambda >0$, 
\begin{eqnarray*}
M_f^{\lambda}(t)&=&e^{-\lambda t}f(X(t))-f(X(0))+\int_0^te^{-\lambda 
s}(\lambda f(X(s))-Af(X(s)))ds\\
&&\qquad\qquad\qquad -\sum_{k=1}^m\int_0^te^{-\lambda s}B_kf(X(s-
))d\gamma_k(s)\end{eqnarray*}
is a $\{{\cal F}_t\}$-martingale. 
\end{lemma}

\begin{proof}
Since ${\cal D}\subset C(E)$, $M_f$ is cadlag, we can apply It\^o's formula 
to 
$e^{-\lambda t}f(X(t))$ and obtain 
\begin{eqnarray*}
e^{-\lambda t}f(X(t))-f(X(0))&=&\int_0^t(-f(X(s))\lambda e^{-\lambda 
s}+e^{-\lambda s}Af(X(s)))ds\\
&&\qquad\qquad +\sum_{k=1}^m\int_0^te^{-\lambda s}B_kf(X(s-))d\gamma_
k(s)+\int_0^te^{-\lambda s}dM_f(s).\end{eqnarray*}
\end{proof}

Lemma \ref{rnlem} is replaced by Lemma 
\ref{ctrans} below.

\begin{lemma}\label{ctrans}

\begin{itemize}

\item[a)]The set of distributions of 
solutions of the constrained martingale problem for 
$(A,E_0;B_1,E_1;\ldots ;B_m,E_m)$ is convex.

\item[b)]Let $X$, $\gamma_1,\ldots ,\gamma_m$ satisfy Definition 
\ref{cmgp}.  Let $\tau\geq 0$ be a bounded $\{{\cal F}_t\}$-stopping time and 
$H\geq 0$ be a ${\cal F}_{\tau}$-measurable random variable such that 
$0<E[H]<\infty$. Then the measure $P^{\tau ,H}\in {\cal P}(D_E[0,
\infty )$ defined by 
\begin{equation}P^{\tau ,H}(C)=\frac {E[H\mbox{\rm \{\bf 1}_
C(X(\tau +\cdot ))]}{E[H]},\quad C\in {\cal B}(D_{\overline {E_0}}
[0,\infty )),\label{cphdist}\end{equation}
is the distribution of 
a solution of the constrained martingale problem for 
$(A,E_0;B_1,E_1;$ $\ldots ;B_m,E_m)$.

\end{itemize}
\end{lemma}

\begin{proof}
Part (a) is immediate.  For Part (b), let $(\Omega ,{\cal F},\mathbb{
P})$ 
be the probability space on which $X$, $\gamma_1,\ldots ,\gamma_m$ are 
defined, and define $\mathbb{P}^H$ by
\[\mathbb{P}^H(C)=\frac {E^{\mathbb{P}}[H{\bf 1}_C]}{E^{\mathbb{P}}
[H]},\quad C\in {\cal F}.\]
Define  $X^{\tau}$ and $\gamma_k^{\tau}$ by $X^{\tau}(t)=X(\tau +
t)$ and 
$\gamma_k^{\tau}(t)=\gamma_k(\tau +t)-\gamma_k(\tau )$. 
$X^{\tau}$ and the $\gamma_k^{\tau}$ are adapted to the filtration $
\{{\cal F}_{\tau +t}\}$ and 
for  $0\leq t_1<\cdots <t_n<t_{n+1}$ and  $f_1,\cdots ,f_n\in B(E
)$,
\begin{eqnarray*}
&&E^{\mathbb{P}^H}\Big[\Big\{f(X^{\tau}(t_{n+1}))-f(X^{\tau}(t_n)
)-\int_{t_n}^{t_{n+1}}Af(X^{\tau}(s))ds-\sum_{k=1}^m\int_{t_n}^{t_{
n+1}}B_kf(X^{\tau}(s-))d\gamma^{\tau}_k(s)\Big\}\\
&&\qquad\qquad\qquad\qquad\qquad\qquad\qquad\qquad\Pi_{i=1}^nf_i(
X^{\tau}(t_i))\Big]\\
&&\quad =\frac 1{E^{\mathbb{P}}[H]}E^{\mathbb{P}}\Big[H\Big\{f(X(
\tau +t_{n+1}))-f(X(\tau +t_n))\\
&&\qquad\qquad\qquad\qquad -\int_{\tau +t_n}^{\tau +t_{n+1}}Af(X(
s))ds-\sum_{k=1}^m\int_{\tau +t_n}^{\tau +t_{n+1}}B_kf(X(s-))d\gamma_
k(s)\Big\}\\
&&\qquad\qquad\qquad\qquad\qquad\qquad\qquad\qquad\Pi_{i=1}^nf_i(
X(\tau +t_i))\Big]\\
&&\quad =0\end{eqnarray*}
by the optional sampling theorem.
Therefore, under $\mathbb{P}^H$, $X^{\tau}$ is a solution of the constrained 
martingale problem with local times $\gamma^{\tau}_1,\cdots ,\gamma^{
\tau}_m$. 
$P^{\tau ,H}$, given by (\ref{cphdist}),
is the distribution of $X^{\tau}$ on $D_{\overline {E_0}}[0,\infty 
)$.
\end{proof}

As in Section \ref{mgpunq}, let $\Pi$ denote the set of 
distributions of solutions of the constrained martingale 
problem and  
$\Pi_{\mu}$ denote the set of 
distributions of solutions with initial condition $\mu$. 
In the rest of this 
section $X$ is the canonical process on $D_E[0,\infty )$ 
and $\gamma_1,\ldots ,\gamma_m$ are a set of $\{{\cal F}^X_t\}$-adapted local times (see 
Remark \ref{adap}). 
We assume that the following conditions hold. 
See Section \ref{suffsect} below for 
settings in which these conditions are valid. Recall that 
we are assuming $E$ is compact.

\begin{condition}\label{Hc}\hfill

\begin{itemize}

\item[a)] ${\cal D}$ is dense in $C(\bar {E}_0)$ in the topology of 
uniform convergence.

\item[b)] For each $\mu\in {\cal P}(\bar {E}_0)$, $\Pi_{\mu}\neq\emptyset$ (see Proposition 
\ref{suffHc}). 

\item[c)] $\Pi$ is compact 
(see Proposition \ref{suffHc}). 

\item[d)] For each $P\in\Pi$ and $\lambda >0$, there exist $\gamma_
1,\ldots ,\gamma_m$ 
satisfying the requirements of Definition \ref{cmgp} 
such that 
 $E^P\Big[\int_0^{\infty}e^{-\lambda t}d\gamma_k(t)\Big]<\infty$, $
k=1,\cdots ,m$ (see Proposition \ref{suffHc}).
\end{itemize}
\end{condition}

\begin{remark}
For $P\in\Pi$, Condition \ref{Hc}(d) and Lemma \ref{cres} give 
\begin{equation}\mu f=E[\int_0^{\infty}e^{-\lambda s}(\lambda f(X
(s))-Af(X(s)))ds-\sum_{k=1}^m\int_0^{\infty}e^{-\lambda s}B_kf(X(
s-))d\gamma_k(s)].\label{cresid}\end{equation}
\end{remark}\medskip

\begin{remark}
We can take the topology on $D_E[0,\infty )$ to be either the 
Skorohod topology or the Jakubowski topology (see 
Remark \ref{Jak}).
\end{remark}\medskip

The definitions of $u_{+}$, $u_{-}$, $\pi_{+}$ and $\pi_{-}$ 
are still given by 
(\ref{uplus}), (\ref{uminus}), (\ref{piplus}) and 
(\ref{piminus}). With Condition \ref{Hc} replacing 
Condition \ref{H}, the proof of Lemma \ref{intrep} carries 
over (Lemma \ref{ctrans} above guarantees that 
Lemmas 4.5.8 and 4.5.10 in \cite{EK86} can be applied).

\begin{lemma}\label{csubsol}Assume 
Condition \ref{Hc} holds.  Then $u_{+}$ is a viscosity 
subsolution of $(\ref{creso})$ and $u_{-}$ is a viscosity 
supersolution of the same equation.  
\end{lemma}

\begin{proof}
The proof is similar to the proof of Lemma \ref{subsol}, 
so we will only sketch the argument.
For $f\in {\cal D}$, let $x_0$ satisfy
$\sup_{x\in\bar {E}_0}(u_{+}-f)(x)=u_{+}(x_0)-f(x_0)$. By adding a constant to 
$f$ if necessary, we 
can assume that $u_{+}(x_0)-f(x_0)=0$.

With $\tau_{\epsilon}$ as in the proof of
Lemma \ref{subsol}, by Lemmas \ref{cres}, \ref{ctrans} 
and \ref{intrep}, and the compactness of $\Pi_{x_0}$ (Condition 
\ref{Hc}(c)), for some $P\in\Pi_{x_0}$ (independent of $\epsilon$) we have 
\begin{eqnarray*}
0\leq E^P\bigg[\int_0^{\tau_{\epsilon}}e^{-\lambda t}\left(h(X(t)
)-\lambda f(X(t))+Af(X(t))\right)dt\\
+\sum_{k=1}^m\int_0^{\tau_{\epsilon}}e^{-\lambda t}B_kf(X(t-))d\gamma_
k(t)\bigg],\end{eqnarray*}
Dividing the 
expectations by $E^P[\lambda^{-1}(1-e^{-\lambda\tau_{\epsilon}})+
\sum_{k=1}^m\int_0^{\tau_{\epsilon}}e^{-\lambda t}d\gamma_k(t)]$ and 
letting $\epsilon\rightarrow 0$, we must have 
\[0\leq (h(x_0)-\lambda f(x_0)+Af(x_0))\vee\max_{k:x_0\in\bar {E}_
k}B_kf(x_0),\]
which, since $f(x_0)=u_{+}(x_0)$, implies (\ref{bndry}), if 
$x_0\in\partial E_0$, and (\ref{interior}), if $x_0\in E_0$.
\end{proof}

\begin{corollary}\label{csol} 
Let $h\in C(\bar {E}_0)$.  If, in addition to Condition \ref{Hc},
the comparison principle holds for 
equation (\ref{creso}), then 
$u=u_{+}=u_{-}$ is the unique viscosity solution of equation 
(\ref{creso}).  
\end{corollary}

The following theorem is the analog of Theorem \ref{mgpb}.

\begin{theorem}\label{cmgpb} 
Assume Condition \ref{Hc} holds.  For $\lambda >0$, let $M_{\lambda}$ 
be the set of $h\in C(\bar {E}_0)$ such that the comparison 
principle holds for (\ref{creso}).  If for every $\lambda >0$, 
$M_{\lambda}$ is separating, then the distribution of the 
solution $X$ of the 
constrained martingale problem for $(A,E_0;B_1,E_1;\ldots ;B_m,E_
m)$ 
is uniquely determined.  
\end{theorem}

\begin{proof}\ 
The proof of this result is in the spirit of Remark 
\ref{extunq}. 
Let $\hat {A}$ be the collection of $(f,g)\in B(\bar {E}_0)\times 
B(\bar {E}_0)$ such that 
$f(X(t))-\int_0^tg(X(s))ds$ is a $\left\{{\cal F}^X_t\right\}$-martingale for all $
P\in\Pi$. 
Denote by $\hat{\Pi}$ the set of the distributions of solutions of 
the martingale problem for $\hat {A}$, and by $\hat{\Pi}_{\mu}$ the set of 
solutions with initial distribution $\mu$. Then, by construction, 
for each $\mu\in {\cal P}(\bar {E}_0)$, $\Pi_{\mu}\subseteq\hat{\Pi}_{
\mu}$.  By the comparison principle, 
Lemmas \ref{csubsol},  
\ref{intrep} and \ref{ext}, for each $h\in M_{\lambda}$ and 
$u=u_{+}=u_{-}$ given by (\ref{uplus}) (or equivalently, 
(\ref{uminus})), 
$(u,\lambda u-h)$ belongs to $\hat {A}$, or equivalently the pair $
(u,h)$ 
belongs to $\lambda -\hat {A}$.  Consequently ${\cal R}(\lambda -
\hat {A})\supseteq M_{\lambda}$ is separating 
and the thesis follows from Corollary \ref{rngunq}.  
\end{proof}

\begin{remark}\label{0bdry}
Differently from Remark \ref{extunq}, the operator $\hat {A}$ is not 
an extension of $A$ 
as an operator on the domain ${\cal D}$, but it is an extension of $
A$ 
restricted to the domain 
${\cal D}_0=\{f\in {\cal D}:\,B_kf(x)=0\,\,\,\forall x\in\bar {E}_
k,\,k=1,\cdots ,m\}$. 
The distribution of the solution $X$ of the 
constrained martingale problem for 
$(A,E_0;B_1,E_1;\ldots ;B_m,E_m,\mu )$ 
is uniquely determined even though the same might not 
hold for the solution of the
martingale problem for $(A\Big|_{{\cal D}_0},\mu )$. 
\end{remark}\medskip

\subsection{Sufficient conditions for the validity of 
Condition \ref{Hc}}\label{suffsect}

In what follows, we assume that $E-E_0=\cup_{k=1}^m\tilde {E}_k$, where 
the $\tilde {E}_k$ are disjoint Borel sets satisfying $\tilde {E}_
k\supset E_k$, 
$k=1,\ldots ,m$.

\begin{proposition}\label{suffHc} 
Assume Condition \ref{Hc}(a)  and that the following 
hold:  

\begin{itemize}

\item[i)] There exist linear operators 
$\tilde {A},\tilde {B}_1,\cdots ,\tilde {B}_m:\tilde {{\cal D}}\subseteq 
C(E)\rightarrow C(E)$ with $\tilde {{\cal D}}$ dense in $C(E)$, $
(1,0)\in\tilde {A}$, 
$(1,0)\in\tilde {B}_k$, $k=1,...,m$, that are extensions of $A,B_
1,\ldots ,B_m$ 
in the sense that for every $f\in {\cal D}$ there exists $\tilde {
f}\in\tilde {{\cal D}}$ such 
that $f=\left.\tilde f\right|_{\bar {E}_0}$, $Af=\left.\tilde A\tilde 
f\right|_{\bar {E}_0}$, and $B_kf=\left.\tilde B_k\tilde f\right|_{
\bar {E}_0}$, 
$k=1,...,m$, and such that the martingale problem for 
each of $\tilde {A},\tilde {B}_1,\cdots ,\tilde {B}_m$ with initial condition $
\delta_x$ has a 
solution for every $x\in E$.  

\item[ii)] If $E\neq\bar {E}_0$, there exists $\varphi\in\tilde {{\cal D}}$ such that 
$\varphi =0$ on $\bar {E}_0$, $\varphi >0$ on $E-\bar {E}_0$ and $
\tilde {A}\varphi =0$ on $\bar {E}_0$, $\tilde {B}_k\varphi\leq 0$ on 
$\overline {\tilde {E}_k}$, $k=1,...,m$. 

\item[iii)] There exists $\left\{\varphi_n\right\}$, $\varphi_n\in 
{\cal D}$, such that 
$\sup_{n,x}|\varphi_n(x)|<\infty$ and $B_k\varphi_n(x)\geq n$ on $
\bar {E}_k$ for all $k=1,...,m$.

\end{itemize}

Then b) c) and d) in Condition \ref{Hc} are verified. 

\end{proposition}

\begin{proof}
{\bf Condition \ref{Hc}(b)}.  We will obtain a solution of the 
constrained martingale problem for 
$(A,E_0;B_1,E_1;\ldots ;B_m,E_m)$ by constructing a solution of the 
constrained martingale problem for 
$(\tilde {A},E_0;\tilde {B}_1,\tilde {E}_1;\ldots ;\tilde {B}_m,\tilde {
E}_m)$ and showing that any such 
solution that starts in $\bar {E}_0$ stays in $\bar {E}_0$ for all times.  
Following \cite{Kur90}, we will construct a solution of 
the constrained martingale problem for 
$(\tilde {A},E_0;\tilde {B}_1,\tilde {E}_1;\ldots ;\tilde {B}_m,\tilde {
E}_m)$ from a solution of the 
corresponding patchwork martingale problem. 

$\tilde {A},\tilde {B}_1,\cdots ,\tilde {B}_m$ are dissipative operators by i) and Lemma 
\ref{diss}. 
Then, by Lemma 1.1 in \cite{Kur90}, for each initial distribution 
on $E$, there exists a solution of the 
patchwork martingale problem for 
$(\tilde {A},E_0;\tilde {B}_1,\tilde {E}_1;\ldots ;\tilde {B}_m,\tilde {
E}_m)$. In addition, if $E\neq\bar {E}_0$, by ii) and the same 
argument used in the proof of Lemma 1.4 in 
\cite{Kur90}, for every solution $Y$ of 
the patchwork martingale problem for 
$(\tilde {A},E_0;\tilde {B}_1,\tilde {E}_1;\ldots ;\tilde {B}_m,\tilde {
E}_m)$ with initial distribution 
concentrated on $\bar {E}_0$,  $Y(t)\in\bar {E}_0$ for all $t\geq 
0$. 
Therefore $Y$ is also a solution of 
the patchwork martingale problem for 
$(A,E_0;B_1,E_1;\ldots ;B_m,E_m)$.
By iii) and Lemma 1.8, Lemma 1.9, 
Proposition 2.2, and Proposition 2.3 in \cite{Kur90}, 
from $Y$, a solution $X$ of the constrained 
martingale problem for $(A,E_0;B_1,E_1;\ldots ;B_m,E_m)$
can be constructed.

\vspace{.1in}

\noindent {\bf Condition \ref{Hc}(c)}.  If $X$ is a solution of the 
constrained martingale problem for 
$(A,E_0;$ $B_1,E_1;\ldots ;B_m,E_m)$ 
and $\gamma_1,\ldots ,\gamma_m$ are associated local 
times, then $\eta_0(t)=\inf\{s:\,s+\,\gamma_1(s)+\ldots +\gamma_m
(s)>t\}$ is strictly 
increasing and diverging to infinity as $t$ goes to infinity, 
with probability one, and $Y=X\circ\eta_0$ is a solution of the 
patchwork martingale problem for 
$(A,E_0;B_1,E_1;\ldots ;B_m,E_m)$, $\eta_0,\eta_1=\gamma_1\circ\eta_
0,\ldots ,\eta_m=\gamma_m\circ\eta_0$ are 
associated increasing processes (see the proof of 
Corollary 2.5 of \cite{Kur90}).  Let $\left\{(X^n,\gamma_1^n,\ldots 
,\gamma_m^n)\right\}$ be a 
sequence of solutions of the constrained martingale 
problem for $(A,E_0;B_1,E_1;\ldots ;B_m,E_m)$ with initial 
conditions $\left\{\mu^n\right\}$, $\mu^n\in {\cal P}(E)$, with associated local times.  
Since ${\cal P}(E)$ is compact, we may assume, without loss of 
generality, that $\left\{\mu^n\right\}$ converges to $\mu$.  Let 
$\left\{(Y^n,\eta_0^n,\eta_1^n,\ldots ,\eta_m^n)\right\}$ be the sequence of the corresponding 
solutions of the patchwork martingale problem and 
associated increasing processes.  Then by the density of ${\cal D}$ 
and Theorems 3.9.1 and 3.9.4 of \cite{EK86},
$\left\{(Y^n,\eta_0^n,\eta_1^n,\ldots ,\eta_m^n)\right\}$ is relatively compact under the 
Skorohod topology on 
$D_{E\times {\bf R}^{m+1}}[0,\infty )$. 

Let $\left\{(Y^{n_k},\eta_0^{n_k},\eta_1^{n_k},\ldots ,\eta_m^{n_
k})\right\}$ be a subsequence converging 
to a limit $(Y,\eta_0,\eta_1,\ldots ,\eta_m)$.  Then $Y$ is a solution of the 
patchwork martingale problem for $(A,E_0;B_1,E_1;\ldots ;B_m,E_m)$ 
with initial condition $\mu$ and $\eta_0,\eta_1,\ldots ,\eta_m$ are associated 
increasing processes.  By iii) and Lemma 1.8 and Lemma 1.9 
in \cite{Kur90}, $\eta_0$ is strictly increasing and diverging to 
infinity as $t$ goes to infinity, with probability one.  It 
follows that $\left\{(\eta_0^{n_k})^{-1}\right\}$ converges to $(
\eta_0)^{-1}$ and hence 
$\left\{(X^{n_k},\gamma_1^{n_k},\ldots ,\gamma_m^{n_k})\right\}=\left
\{(Y^{n_k}\circ (\eta_0^{n_k})^{-1},\eta_1^{n_k}\circ (\eta_0^{n_
k})^{-1},\ldots ,\eta_m^{n_k}\circ (\eta_0^{n_k})^{-1})\right\}$ 
converges to $(Y\circ (\eta_0)^{-1},\eta_1\circ (\eta_0)^{-1},\ldots 
,\eta_m\circ (\eta_0)^{-1})$ and 
$Y\circ (\eta_0)^{-1}$ with associated local times 
$\eta_1\circ (\eta_0)^{-1},\ldots ,\eta_m\circ (\eta_0)^{-1}$ is a solution of the constrained 
martingale problem for $(A,E_0;B_1,E_1;\ldots ;B_m,E_m)$ with 
initial condition $\mu$.  

\vspace{.1in}

\noindent {\bf Condition \ref{Hc}(d)}.  Let $\varphi_1$ be the function of 
iii) for $n=1$.  By Lemma \ref{cres} and iii) we have 
\[\begin{array}{l}E\Big[\sum_{k=1}^m\int_0^te^{-\lambda s}d\gamma_k(s)\Big]\\
\qquad\leq e^{-\lambda t}E[\varphi_1(X(t))]-E[\varphi_1(X(0))]+\int_0^te^{-\lambda 
s}E[\lambda\varphi_1(X(s))-A\varphi_1(X(s))]ds.
\end{array}\]
\end{proof}

\section{Examples}\label{examples} 

Several examples of application of the results of the 
previous sections can be given by exploiting comparison 
principles proved in the literature. Here we will discuss 
in detail two examples. 

The first example is a class of diffusion processes 
reflecting in a domain $D\subseteq {\bf R}^d$ according to an oblique direction 
of reflection which may become tangential. 
This case is not covered by the existing literature on reflecting 
diffusions, which assumes that the direction of reflection is 
uniformly bounded away from the tangent hyperplane. 

The second example is a large class of jump diffusion processes 
with jump component of 
unbounded variation and possibly degenerate diffusion 
matrix. In this case uniqueness results are already 
available in the literature (see e.g. \cite{IW89}, 
\cite{Gra92}, \cite{KP96}) but we believe it is 
still a good benchmark to show how our method works.

\subsection{Diffusions with degenerate oblique 
direction of reflection}\label{degref}

Let $D\subseteq {\bf R}^d$, $d\geq 2$, be a bounded domain with $
C^3$ boundary, 
i.e. 
\[\begin{array}{cc}
D=\{x\in {\bf R}^d:\,\psi (x)>0\},\quad\partial D=\{x\in {\bf R}^
d:\,\psi (x)=0\},\\
|\nabla\psi (x)|\,>0,\quad\mbox{\rm for }x\in\partial D,\end{array}
\]
for some function $\psi\in C^3({\bf R}^d)$, where $\nabla$ denotes the 
gradient, viewed as a row vector. Let 
$l:\bar {D}\rightarrow {\bf R}^d$ be a vector field in $C^2(\bar {
D})$ such that 
\begin{equation}|l(x)|>0\mbox{\rm \ and }\langle l(x),\nu (x)\rangle
\geq 0,\quad\forall x\in\partial D,\label{drdir1}\end{equation}
$\nu$ being the unit inward normal vector field, and let  
\begin{equation}\partial_0D=\{x\in\partial D:\,\langle l(x),\nu (
x)\rangle =0\}.\label{drdir2}\end{equation}
We assume that $\partial_0D$ has dimension $d-2$. More precisely, 
for $d\geq 3$, we assume that $\partial_0D$ has a finite number of connected 
components, each of the form 
\begin{equation}\{x\in\partial D:\,\psi (x)=0,\tilde{\psi }(x)=0\}
,\label{drdir3}\end{equation}
where $\psi$ is the function above and $\tilde{\psi}$ is another function 
in $C^2({\bf R}^d)$ such that the level set $\{x\in\partial D:\,\tilde{
\psi }(x)=0\}$ 
is bounded and $|\nabla\tilde{\psi }(x)|\,>0$ on it. 
For $d=2$, we assume that $\partial_0D$ consists of a finite 
number of points.  
In addition, we assume that $l(x)$ is never tangential to $\partial_
0D$.  

Our goal is to prove uniqueness of the reflecting 
diffusion process with generator of the form 
\begin{equation}Af(x)=\frac 12tr\left(D^2f(x)\sigma (x)\sigma (x)^
T\right)+\nabla f(x)b(x),\label{drop}\end{equation}
where $\sigma$ and $b$ are Lipschitz continuous 
functions on $\bar {D}$, and direction of reflection $l$.  
We will characterize this reflecting diffusion process as 
the unique solution of the constrained martingale 
problem for $(A,D;B,\partial D)$, where $A$ is given by 
(\ref{drop}), 
\begin{equation}Bf(x)=\nabla f(x)l(x),\label{drbop}\end{equation}
and the common domain of $A$ and $B$ is ${\cal D}=C^2(\overline D
)$. Our 
tools will be the results of Section \ref{cmgpunq} and the 
comparison principle proved  by
\cite{PK05}.

\begin{proposition}
Condition \ref{Hc} is verified. 
\end{proposition}

\begin{proof}
Condition \ref{Hc}a) is obviously verified. Therefore we 
only need to prove that the assumptions of Proposition 
\ref{suffHc} are satisfied. 
Let $0<r<1$ be small enough that for $d(x,\partial D)<\frac 43r$  
the normal projection of $x$ on $\partial D$, $\pi_{\nu}(x)$, is well 
defined and $|\nabla\psi (x)|\,>0$. Set $U(\bar {D})=\left\{x:\,d(x,\bar 
D)<r\right\}$. 
Let $\chi (c)$ be a nondecreasing function in $C^{\infty}({\bf R}
)$ such that  
$0\leq\chi (c)\leq 1$, $\chi (c)=1$ for $c\geq\frac {2r}3$, $\chi 
(c)=0$ for $c\leq\frac r3$. 
We can extend $l$ to a Lipschitz continuous vector field on $\overline {
U(\bar {D})}$ 
by setting, for $x\in$$\overline {U(\bar {D})}-\bar {D}$, 
\[\begin{array}{cc}
l(x)=(1-\chi (d(x,\partial D)))\,l(\pi_{\nu}(x)).\end{array}
\]
We can also extend $\sigma$ and $b$ to Lipschitz continuous 
functions on $\overline {U(\bar {D})}$ by setting, for $x\in$$\overline {
U(\bar {D})}-\bar {D}$, 
\[\begin{array}{cc}
\sigma (x)=(1-\chi (d(x,\partial D)))\,\sigma (\pi_{\nu}(x)),\\
b(x)=(1-\chi (d(x,\partial D)))\,b(\pi_{\nu}(x)).\end{array}
\]
Clearly, both the martingale problem for $A$, with domain 
$C^2(\overline {U(\bar {D})})$, and the martingale problem for $B$, with the 
same domain, have a solution for every initial condition 
$\delta_x$, $x\in\overline {U(\bar {D})}$. Since every $f\in C^2(\bar 
D)$ 
can be extended to a function $\tilde {f}\in C^2(\overline {U(\bar {
D})})$ and 
\[Af=\left.\left(A\tilde f\right)\right|_{\bar {D}},\quad Bf=\left
.\left(B\tilde f\right)\right|_{\bar {D}},\]
Condition (i) in Proposition \ref{suffHc} is verified. 

Next, consider the function $\varphi$ defined as 
\[\varphi (x)\left\{\begin{array}{ll}
=0,&\mbox{\rm for }x\in\bar {D},\\
=\exp\{\frac {-1}{d(x,\partial D)}\},&\mbox{\rm for }x\in\overline {
U(\bar {D})}-\bar {D},\end{array}
\right.\]
where $U(\bar {D})$ is as above. Since $\partial D$ is of class $
C^3$, 
$\varphi\in C^2(\overline {U(\bar {D})})$. Moreover 
\[\nabla\varphi (x)=-\,|\nabla\varphi (x)|\,\nu (\pi_{\nu}(x)),\quad\mbox{\rm for }
x\in\overline {U(\bar {D})}-\bar {D}.\]
Therefore $\varphi$ satisfies Condition (ii) in Proposition \ref{suffHc}.

Finally, in order to verify iii) of Proposition 
\ref{suffHc}, we just need to modify slightly the proof of 
Lemma 3.1 in \cite{PK05}. Suppose first that $\partial_0D$ is 
connected. Let $\tilde{\psi}$ be the function in $(\ref{drdir3})$. Since 
$l(x)$ is never tangent to $\partial_0D$, it must hold $\nabla\tilde{
\psi }(x)l(x)\neq 0$ 
for each $x\in\partial_0D$, 
and hence, possibly replacing $\psi$ by $-\psi$, we can assume 
that 
\begin{equation}\tilde{\psi }(x)=0,\quad\nabla\tilde{\psi }(x)l(x
)>0,\quad\quad\forall\,x\in\partial_0D.\label{skD}\end{equation}
Let $U(\partial_0D)$ be a neighborhood of $\partial_0D$ such that 
$\inf_{\overline {U(\partial_0D)}}\nabla\tilde{\psi }(x)l(x)>0$, and for each $
n\in {\bf N}$, set 
\[\begin{array}{c}
\partial_0^nD=\Big\{x\in\partial D\cap\overline {U(\partial_0D)}:\,\,
|\tilde{\psi }(x)|<\frac 1{2n}\Big\},\\
\tilde {C}_n=\frac 1{\inf_{\partial_0^nD}\nabla\tilde{\psi }(x)l(x)},\\
C_n=\frac {\tilde {C}_n\sup_{\partial D}|\nabla\tilde{\psi }(x)l(x)|+1}{\inf_{
\partial D-\partial_0^nD}\nabla\psi (x)l(x)}.\end{array}
\]
Let $\chi_n$ be a function in $C^{\infty}({\bf R})$ such that $\chi_
n(c)=nc$ 
for $|c|\leq\frac 1{2n}$, $\chi_n(c)=-1$ for $c\leq -\frac 1n$, 
$\chi_n(c)=1$ for $c\geq\frac 1n$, $0\leq\chi_n'(c)\leq n$ for every $
c\in {\bf R}$, and define 
\[\varphi_n(x)=\chi_n(C_n\psi (x))+\tilde {C}_n\chi_n(\tilde{\psi }
(x)).\]
Then $|\varphi_n(x)|$ is bounded by $1+\frac 1{\inf_{\partial_0D}
\nabla\tilde{\psi }(x)l(x)}$ and we 
have, for $x\in\partial_0^nD$, 
\[\nabla\varphi_n(x)l(x)=n\Big[C_n\nabla\psi (x)l(x)+\tilde {C}_n\nabla\tilde{\psi }
(x)l(x)\Big]\geq n,\]
and for $x\in\partial D-\partial_0^nD$, 
\[\nabla\varphi_n(x)l(x)=nC_n\nabla\psi (x)l(x)+\tilde {C}_n\chi_n'(\tilde{
\psi }(x))\nabla\tilde{\psi }(x)l(x)\geq n.\]

If $\partial_0D$ is not connected, there is a function $\tilde{\psi}^
k$ 
satisfying $(\ref{skD})$ for each connected component 
$\partial_0^kD$. Let $U^k(\partial_0^kD)$ be neighborhoods such that 
$\inf_{\overline {U^k(\partial_0^kD)}}\nabla\tilde{\psi}^k(x)l(x)>0$. We can assume, without loss of generality, that 
$\overline {U^k(\partial_0^kD)}\subseteq V^k(\partial_0^kD)$, where $\overline {
V^k(\partial_0^kD)}$ are 
pairwise disjoint and $\tilde{\psi}^k$ vanishes outside $V^k(\partial_
0^kD)$. 
Then, defining $\partial_0^{k,n}D$ and $\tilde {C}_n^k$ as above,  
\[C_n^k=\frac {\tilde {C}_n^k\sup_{\partial D}|\nabla\tilde{\psi}^k(x)
l(x)|+1}{\inf_{\partial D-\cup_k\partial_0^{k,n}D}\nabla\psi (x)l(x)},\]
and $\varphi_n^k$ as above, $\varphi_n(x)=\sum_k\varphi_n^k(x)$ verifies 
iii) of Proposition \ref{suffHc}.
\end{proof}. 

Theorem 2.6 of \cite{PK05} gives the comparison 
principle for a class of linear and nonlinear equations 
that includes, in particular, the partial differential 
equation with boundary conditions 
\begin{equation}\begin{array}{ll}
\lambda u(x)-Au(x)=h(x),&\mbox{\rm in }D,\\
-Bu(x)=0,&\mbox{\rm on }\partial D,\end{array}
\label{drpde}\end{equation}
where $h$ is a Lipschitz continuous function, and $A$, $B$ are 
given by (\ref{drop}), (\ref{drbop}) and verify, in 
addition to the the assumptions formulated at the 
beginning of this section, the following local condition on 
$\partial_0D$.  

\begin{condition}\label{PK}\hfill

For every $x_0\in\partial_0D$, let $\phi$ be a $C^2$ diffeomorphism from 
the closure of a 
suitable neighborhood $V$ of the origin into the closure of a suitable 
neighborhood of $x_0$, $U(x_0)$, such that $\phi (0)=x_0$ and the 
$d\mbox{\rm th}$ column of $J\phi (z)$, $J_d\phi (z)$, satisfies 
\begin{equation}J_d\phi (z)=-l(\phi (z)),\quad\forall z\in\phi^{-
1}\left(\overline {\partial D\cap U(x_0)}\right).\label{drchc}\end{equation}
Let $\tilde {A}$, 
\[\tilde {A}f(z)=\frac 12tr\left(D^2f(z)\tilde\sigma (z)\tilde\sigma^
T(z)\right)+\nabla f(z)\tilde {b}(z),\]
be the operator such that 
\[\tilde {A}(f\circ\phi )(z)=Af(\phi (z)),\quad\forall z\in\phi^{
-1}\left(\overline {D\cap U(x_0)}\right).\]
Assume 
\begin{itemize}

\item[a)] $\tilde {b}^i$, $i=1,...,d-1$, is a function of the first $
d-1$ 
coordinates $(z^1,...,z^{d-1})$ only, and $\tilde {b}_d$ is a function of $
z_d$ 
only.  

\item[b)] $\tilde{\sigma}^{ij}$, $i=1,...,d-1$, $j=1,...,d$ is a function of the 
first $d-1$ coordinates $(z^1,...,z^{d-1})$ only.  
\end{itemize}
\end{condition}

\begin{remark}\label{drnorm}
For every $x_0\in\partial_0D$, some coordinate of $l(x_0)$, say the 
$d\mbox{\rm th}$ coordinate, must be nonzero. Then in 
(\ref{drchc}) we can choose $U(x_0)$ such that in $\overline {
U(x_0)}$ 
$l^d(x)\neq 0$ and we can replace $l(x)$ by $l(x)/|l^d(x)|$, 
since this normalization does not change the boundary 
condition of (\ref{drpde}) in $\overline {D\cap U(x_0)}$ (i.e. any 
viscosity sub/supersolution of (\ref{drpde}) in $\overline {D\cap 
U(x_0)}$ 
is a viscosity sub/supersolution of (\ref{drpde}) in $\overline {
D\cap U(x_0)}$ with 
the normalized vector field and conversely).  

Moreover, since (\ref{drchc}) must be verified only in 
$\phi^{-1}\left(\overline {\partial D\cap U(x_0)}\right)$, in the construction of $
\phi$ we can use any 
$C^2$ vector field $\bar {l}$ that agrees with $l$, or the above 
normalization of $l,$ on $\overline {\partial D\cap U(x_0)}$.  
\end{remark}\medskip

Therefore, whenever Condition \ref{PK} is satisfied 
Theorem \ref{cmgpb} applies and there exists one and only 
one diffusion process reflecting in $D$ according to the 
degenerate oblique direction of reflection $l$. 

The following is a concrete example 
where Condition \ref{PK} is satisfied.

\begin{example}
{\rm Let 
\[D=B_1(0)\subseteq {\bf R}^2.\]
and suppose the direction of reflection $l$ satisfies 
(\ref{drdir1}) with the strict inequality at 
every $x\in\partial D$ except at $x_0=(1,0)$, where 
\[l(1,0)=\left[\begin{array}{c}
0\\
-1\end{array}
\right].\]
Of course, in a neighborhood of $x_0=(1,0)$ we can always 
assume that $l$ depends only on the second coordinate $x_2$. 
In addition, by Remark \ref{drnorm}, we can suppose   
\[l_2(x)=-1.\]
Consider
\[\sigma (x)=\sigma =\left[\begin{array}{cc}
1&0\\
0&0\end{array}
\right].\]
and a drift $b$ that, in a neighborhood of $x_0=(1,0)$, 
depends only on the second coordinate, i.e. 
\[b(x)=b(x_2).\]
Assume that, in a neighborhood of $x_0=(1,0)$, the 
direction of reflection $l$ is parallel to $b$. 
Then we can find a nonlinear change of coordinates 
$\phi$ such that Condition \ref{PK} is verified, namely  
\[\begin{array}{c}
\phi_2(z)=z_2\\
\phi_1(z)=-\int_0^{z_2}l_1(\zeta_2)d\zeta_2+z_1+1,\end{array}
\]
which yields  
\[\tilde{\sigma }(z)=\sigma ,\quad\tilde {b}_1(z)=0,\quad\tilde {
b}_2(z)=b_2(z_2).\]

}
\end{example}

\subsection{Jump diffusions with degenerate diffusion 
matrix}\label{jd}.

Consider the operator 
\[Af=Lf+Jf\]
\begin{equation}Lf(x)=\frac 12tr\left(a(x)D^2f(x)\right)+\nabla f
(x)b(x)\label{JD-op}\end{equation}
\[Jf(x)=\int_{{\mathbb R}^{d'}-\{0\}}\left[f(x+\eta (x,z))-f(x)-\nabla 
f(x)\eta (x,z)I_{|z|<1}\right]m(dz),\]
where $\nabla is$ viewed as a row vector.
\noindent
Assume: 

\begin{condition}\label{JD-ex}\hfill 

\begin{itemize}

\item[a)] $a=\sigma\sigma^T$, $\sigma$ and $b$ are continuous.

\item[b)] $\eta (\cdot ,z)$ is continuous for every $z$, $\eta (x,
\cdot )$ is 
Borel measurable for every $x$, 

\noindent
$\sup_{|z|<1}|\eta (x,z)|<+\infty$ for 
every $x$ and 
\[|\eta (x,z)|I_{|z|<1}\leq\rho (z)(1+|x|)I_{|z|<1},\]
for some positive, measurable function $\rho$ such that 
$\lim_{|z|\rightarrow 0}\rho (z)=0$. 

\item[c)] $m$ is a Borel measure such that 
\[\int_{{\mathbb R}^{d'}-\{0\}}\left[\rho (z)^2I_{|z|<1}+I_{|z|\geq 
1}\right]m(dz)<+\infty .\]
\end{itemize}

\end{condition}

\noindent
Then, with ${\cal D}(A)=\{f+c:f\in C^2_c({\mathbb R}^d),c\in {\mathbb R}
\}$, 
$A\subset C_b(E)\times C_b(E)$ and $A$ satisfies Condition \ref{H}.

A comparison principle for bounded subsolutions and 
supersolutions of the equation $(\ref{reso})$ 
when $A$ is given by $(\ref{JD-op})$ is proven in 
\cite{JK06}, as a special case of a more general result, 
under the following assumptions:

\begin{condition}\label{JD-comp}\hfill 

\begin{itemize}
\item[a)] $\sigma$ and $b$ are Lipschitz continuous.

\item[b)]
\[|\eta (x,z)-\eta (y,z)|I_{|z|<1}\leq\rho (z)|x-y|I_{|z|<1}.\]
\item[c)] $h$ is uniformly continuous.
\end{itemize}

\end{condition}

Then, under the above assumptions, our result of Theorem 
\ref{mgpb} applies and uniqueness of the solution of the 
martingale problem for $A$ is granted.

%%%%%%%%%%%%%%%%%%%%%%%%%%%%%%%%%%%%%%%%%%%%%%%%%%%%%%%%%%%%%%%%%%%
%%                                                               %%
%% Use the two commands below for producing your bibliography    %%
%% with bibtex, then comment again the commands and include the  %%
%% content of the .bbl file in this file below the commands.     %%
%%                                                               %%
%%%%%%%%%%%%%%%%%%%%%%%%%%%%%%%%%%%%%%%%%%%%%%%%%%%%%%%%%%%%%%%%%%%

%\bibliographystyle{amsplain}
%\bibliography{viscmgp_TS}

% add below the content of your .bbl file produced by bibtex.

%%%%%%%%%%%%%%%%%%%%%%%%%%%%%%%%%%%%%%%%%%%%%%%%%%%%%%%%%%%%%%%%%%%
%%                                                               %%
%% You may add acknowledgments (optional).                       %%
%%                                                               %%
%%%%%%%%%%%%%%%%%%%%%%%%%%%%%%%%%%%%%%%%%%%%%%%%%%%%%%%%%%%%%%%%%%%

%%\ACKNO{}

%%%%%%%%%%%%%%%%%%%%%%%%%%%%%%%%%%%%%%%%%%%%%%%%%%%%%%%%%%%%%%%%%%%
%%                                                               %%
%% You have reached the end of your document.                    %%
%%                                                               %%
%%%%%%%%%%%%%%%%%%%%%%%%%%%%%%%%%%%%%%%%%%%%%%%%%%%%%%%%%%%%%%%%%%%

\end{document}